\documentclass[12pt]{amsart}
\usepackage{amsmath,amssymb,amsthm,mathrsfs,verbatim}
\usepackage[margin=1.0in]{geometry}

\usepackage{amsmath}
\usepackage{amssymb}
\usepackage{amsfonts}
\usepackage{geometry}
\usepackage{url}

\usepackage{amsmath}
\usepackage{amssymb}
\usepackage{amsfonts}
\usepackage{geometry}
\usepackage{url}
\usepackage{setspace}
\usepackage{amsthm}
\usepackage[all,cmtip]{xy}
\usepackage{amsmath,amscd}

\makeatletter
\def\url@leostyle{%
  \@ifundefined{selectfont}{\def\UrlFont{\sf}}{\def\UrlFont{\small\ttfamily}}}
\makeatother
\newtheorem{thm}{Theorem}[section]
\newtheorem{prop}[thm]{Proposition}
\newtheorem{lemma}[thm]{Lemma}
\newtheorem{cor}[thm]{Corollary}

\theoremstyle{definition}
\newtheorem{defn}[thm]{Definition}

\newtheorem{rem}[thm]{Remark}

\newcommand{\Ker}{\operatorname{Ker}}

\newcommand{\ldot}{\textbf{.}}

\newcommand{\triv}{\mathrm{triv}}
\newcommand{\rank}{\operatorname{rank}}

\newcommand{\mf}{\mathfrak}

\newcommand{\Z}{\mathbb{Z}}

\newcommand{\F}{\mathbb{F}}
\newcommand{\h}{\mf{h}}

\begin{document}

\author{Martina Balagovi\' c and Harrison Chen}
\address{
M.B: Department of Mathematics, University of York, York, YO10 5DD, UK, and \\
Department of Mathematics, University of Zagreb, Bijeni\v{c}ka 30, 10000 Zagreb, Croatia\\
H.C: Department of Mathematics, University of California, 852 Evans Hall, Berkeley, CA 94720 USA}
\email{martina.balagovic@york.ac.uk, chenhi@math.berkeley.edu}
\title{Category $\mathcal{O}$ for Rational Cherednik Algebras $H_{t,c}(GL_2(\F_p),\h)$ in Characteristic $p$}

\begin{abstract}
In this paper we describe the characters of irreducible objects in category $\mathcal{O}$ for the rational Cherednik algebra associated to $GL_2(\F_p)$ over an algebraically closed field of positive characteristic $p$, for any value of the parameter $t$ and generic value of the parameter $c$.
\end{abstract}

\maketitle

\section{Introduction}

This paper is a sequel to \cite{BalChen}, and continues the study of rational Cherednik algebras over algebraically closed fields of positive characteristic. 

Let $p$ be an odd prime, $\Bbbk$ an algebraically closed field of characteristic $p$, $\F_p\subset \Bbbk$ the subfield of $p$ elements, $G=GL_2(\F_p)$ the general linear group over $\F_p$, $\h=\Bbbk^2$ the tautological column vector representation of $G$, $\h_{\F}=\F_p^2$ its $\F_p$-form, and $\h^*$ the dual representation. For $t\in \Bbbk$ a constant, and $c$ a collection of conjugation invariant parameters in $\Bbbk$ labeled by reflections in $G$,  the \emph{rational Cherednik algebra} $H_{t,c}(G,\h)$ is a non-commutative, associative, infinite dimensional algebra deforming the semidirect product of the group algebra $\Bbbk[G]$ and the symmetric algebra $S(\h^{*} \oplus \h)$. Algebras of this type for various reflection groups $G$, have been extensively studied since the early 1990s, mostly over fields of characteristic zero.

We first repeat some general results about rational Cherednik algebras (which can be found in \cite{etingof-ma}), some results specific for finite characteristic (which can be found in \cite{BalChen}),
and some results about the structure of the category of finite dimensional representations of $GL_2(\F_p)$ (which can be found in \cite{brauer}, \cite{dickson} and \cite{glover1}). 
The main theorem, Theorem \ref{mainII}, calculates 
the characters of irreducible $H_{t,c}(GL_2(\F_p),\h)$- modules $L_{t,c}(\tau)$, for any $t$, generic $c$ and all $\tau$. 

The roadmap of this paper is as follows. Section 2 contains definitions and basic properties of rational Cherednik algebras, their representation theory, 
and category $\mathcal{O}$. We state most results without proof and stress the differences between the characteristic zero and characteristic $p$ cases. Section 3 contains
properties of the group $GL_2(\F_p)$ and its representations.
Section 4 contains some preliminary observations and calculations about the category $\mathcal{O}$ for $H_{t,c}(GL_2(\F_p),\h)$, including a reduction which aids in our character computations.
Finally, Sections 5-7 contain the character calculations for irreducible representations associated to representations of $GL_2(\F_p)$.

\section{Rational Cherednik algebras and their representations}\label{definitions}

\subsection{Notation}
Let $\Bbbk$ be an algebraically closed field, $\h$ a $\Bbbk$-vector space, and $G\subseteq GL(\h)$ a finite group generated by the set $S$ of reflections in $G$. An element $s\in G$ is a reflection if $\mathrm{rank}_{\h}(1-s) = 1$. Let $\h^*$ be the dual representation, and $(\cdot,\cdot)$ the canonical pairing $\h \otimes \h^* \rightarrow \Bbbk$ or $\h^* \otimes \h \rightarrow \Bbbk$.  Let $\alpha_s\in \h^*$ be a nonzero element of the image of $1-s$ on $\h^*$. To this data one can associate a rational Cherednik algebra. In this paper, we are interested in the case when $\Bbbk$ is an algebraically closed field of characteristic $p$, $G=GL_2(\F_p)$, and $\h=\Bbbk^2$ is its vector representation. 

For a $\Bbbk$-vector space $V$, let $TV$ and $SV$ denote the tensor and symmetric algebra of $V$ over $\Bbbk$, and $S^iV$ the homogeneous subspace of $SV$ of degree $i$. For a graded vector space $M$, let $M_i$ denote the $i$-th graded piece, $M_+$ the subspace of positive graded degrees, and $M[j]$ the same vector space with the grading shifted by $j$, meaning $M[j]_i = M_{i+j}$.

\subsection{Rational Cherednik algebras} The following definitions and results are standard and can be found in \cite{etingof-ma}.

\begin{defn}Let $t\in \Bbbk$, and $s\mapsto c_s$ be a $\Bbbk$-valued function on the set $S$ of reflections in $G$, satisfying $c_{s}=c_{gsg^{-1}}$ for any $g\in G$. 
The \emph{rational Cherednik algebra} $H_{t,c}(G, \h)$ is the quotient of the semidirect product $\Bbbk[G] \ltimes T(\h \oplus \h^*)$ by the ideal generated by relations:
$$[x, x'] = 0, \,\, [y, y'] = 0, \,\, [y, x] = (y, x)t - \sum_{s \in S} c_s ((1 - s)\ldot x, y) s,$$
for all $x,x'\in \h^*,y,y'\in \h.$
\end{defn}

For $g \in G$ and $y \in \h$, we use notation  $gy$ for multiplication in the algebra, and  $g \ldot y$ for the action from the representation; they are related by $gyg^{-1} = g \ldot y$.

\begin{thm} \label{pbw}
\begin{enumerate}
\item The multiplication map $S\h^*\otimes \Bbbk[G]\otimes S\h\to H_{t,c}(G, \h)$ is an isomorphism. (PBW, \cite{griffeth}, Theorem 2.1. ) 
\item $H_{t,c}(G, \h)$ is graded  by $\deg x=1, \deg y=-1, \deg w=0$ for $x\in \h^*, y\in h, w\in G$.
\item For any $a\in \Bbbk^{\times}$ there exists an isomorphism of algebras $H_{t,c}(G, \h)\cong H_{at,ac}(G, \h)$.
\end{enumerate}
\end{thm}

We can assume without loss of generality that $t=0$ or $t=1$. These two cases behave differently, and we study them both.

\subsection{Verma Modules $M_{t,c}(\tau)$ and Dunkl operators} \label{vermaintro}

\begin{defn}
Let $\tau$ be an irreducible finite dimensional representation of $G$. Define a $\Bbbk[G] \ltimes S\h$-module structure on it by requiring  $\h$-action on $\tau$ to be zero. The \emph{Verma module} is the induced $H_{t,c}(G, \h)$-module  $$M_{t,c}(G) = H_{t,c}(G, \h) \otimes_{\Bbbk[G] \ltimes S\h} \tau.$$ 
\end{defn}

\begin{lemma} \label{homexist}
\begin{enumerate}
\item Let $M$ be an $H_{t,c}(G, \h)$-module and $\tau \subset M$ a $G$-submodule such that $\h\subseteq H_{t,c}(G, \h)$ acts on $\tau$ as zero. Then there is a unique $H_{t,c}(G, \h)$-homomorphism $\phi: M_{t,c}(\tau) \rightarrow M$ such that $\phi|_\tau$ is the identity.
\item As a vector space, $M_{t,c}(\tau)\cong S\h^* \otimes \tau$. It is a graded $H_{t,c}(G, \h)$-representation, with the grading given by degrees of $ S\h^*$. Each graded piece of $M_{t,c}(\tau)$ is a finite dimensional $G$-representation.
\item The action of the generators of $H_{t,c}(G,\h)$ on the Verma module is given by:
$$x \ldot (f \otimes v) = (xf) \otimes v$$
$$g \ldot (f \otimes v) = g\ldot f \otimes g \ldot v$$
$$y \ldot (f\otimes v) =D_y(f \otimes v)$$
for  $f \otimes v \in S\h^* \otimes \tau , x\in \h^*,  g\in G, y\in \h$, and $$D_{y}=t \partial_y \otimes 1 - \sum_{s \in S} c_s \frac{(y, \alpha_s)}{\alpha_s} (1 - s) \otimes s$$ the Dunkl operator.
\end{enumerate}
\end{lemma}

We say a homogeneous element $v\in M_{t,c}(\tau)$ is \emph{singular} if $D_yv=0$ for all $y\in \h$. Any such element of positive degree generates a proper $H_{t,c}(G,\h)$ submodule. By Lemma \ref{homexist}, this submodule is isomorphic to a quotient of $M_{t,c}(G\ldot v)$. We want to describe irreducible quotients of Verma modules. 

\subsection{Contravariant Form $B$}\label{formdef}

Define $\bar{c}:S\to \Bbbk$ as $\bar{c}_s = c_{s^{-1}}$. 

\begin{prop}\label{bform}
There exists a unique  form $B:M_{t,c}(\tau) \times M_{t,\bar{c}}(\tau^*) \rightarrow \Bbbk$ satisfying: \begin{itemize}
\item[a)] Contravariance: for $f \in M_{t,c}(\tau)$, $h \in M_{t,\bar{c}}(\tau^*)$, $g\in G, x \in \h^* , y \in \h$,
$$B(g \ldot f, g \ldot h) = B(f, h), B(xf, h) = B(f, D_{x}(h)), B(f, yh) = B(D_{y}(f), h).$$
\item[b)] In the zeroth degree $M_{t,c}(\tau)_0 \otimes M_{t,\bar{c}}(\tau^*)_0$, $B$ is the canonical pairing  $\tau \otimes \tau^*\to \Bbbk$. 
\end{itemize}
Additionally,
\begin{itemize}
\item[c)] $B: M_{t,c}(\tau)_i \otimes M_{t,\bar{c}}(\tau^*)_j$ is the zero map for $i\ne j$.
\item[d)] $B(f,-)=0$ if $f$ is a singular vector of positive degree.
\item[e)] $\Ker B=\{ f \in  M_{t,c}(\tau) : B(f,-)=0\}$ is the maximal proper graded submodule of $ M_{t,c}(\tau)$.
\end{itemize}
\end{prop}

Let $B_i: M_{t,c}(\tau)_i \otimes M_{t,\bar{c}}(\tau^*)_i$ be the restriction of $B$ to the $i$-th graded piece. In the initial stages of this project, we calculated the matrices of $B_i$ using the MAGMA algebra software \cite{magma} to find singular vectors in the Verma module. 

\subsection{Baby Verma modules $N_{t,c}(\tau)$}
 From now on, $\Bbbk$ is algebraically closed of finite characteristic $p$. We shall state some definitions and results specific to finite characteristic. Most of these can be found in \cite{BalChen}.

As usual in characteristic $p$, the algebra $H_{t,c}(G,\h)$ has a large center, and $M_{t,c}(\tau)$ has a large graded submodule. For any $G$-representation $V$, let $V^G$ denote the subspace of $G$-invariants. 
If  $t=0$, then $(S\h^*)^G\oplus (S\h)^G$ is a central subalgebra, and $((S\h^*)^G)_+ M_{0,c}(\tau)$ is a proper graded submodule of $M_{0,c}(\tau)$. If $t\ne 0$, then  $((S\h^*)^G)^p\oplus ((S\h)^G)^p$ is central and $((S\h^*)^G)^p_+ M_{t,c}(\tau)$ is a proper graded submodule of $M_{t,c}(\tau)$. 

\begin{defn}
For $t\ne 0$, the \emph{baby Verma module} $N_{t,c}(\tau)$ for the algebra $H_{t,c}(G,\h)$ is the quotient $$N_{t,c}(\tau)= M_{t,c}(\tau)/((S\h^*)^G)^p_+ M_{t,c}(\tau).$$ 

For $t= 0$, the \emph{baby Verma module} $N_{0,c}(\tau)$ for the algebra $H_{0,c}(G,\h)$ is the quotient $$N_{0,c}(\tau)= M_{0,c}(\tau)/((S\h^*)^G)_+ M_{0,c}(\tau).$$ 
\end{defn}

\begin{prop}
$N_{t,c}(\tau)$ is graded and finite dimensional. The form $B$ descends to it, and the kernel of the induced form $B_N$ on $N_{t,c}(\tau)$ is the unique maximal proper submodule of $N_{t,c}(\tau)$. 
\end{prop}

Baby Verma modules can be thought of as Verma modules for the restricted rational Cherednik algebra $H_{t,c}(G,\h)/((S\h^*)^G)^p\oplus ((S\h)^G)^p$. All irreducible quotients of Verma modules factor through baby Verma modules. They are sometimes more convenient to work with than Verma modules.

\subsection{Irreducible modules $L_{t,c}(\tau)$ and category $\mathcal{O}$}

Let $J_{t,c}(\tau)=\Ker B\subseteq M_{t,c}(\tau)$. 
\begin{prop}
For any $\tau$, the module $$L_{t,c}(\tau)=M_{t,c}(\tau)/J_{t,c}(\tau)\cong M_{t,c}(\tau)/\ker B_N$$
is an irreducible, graded, finite dimensional $H_{t,c}(G,\h)$-module. 
\end{prop}

\begin{defn}
The category $\mathcal{O}=\mathcal{O}_{t,c}(G, \h)$ is the category of $\Z$-graded $H_{t,c}(G, \h)$-modules which are finite dimensional over $\Bbbk$.
\end{defn}

\begin{prop}
The irreducible objects in $\mathcal{O}$ are $L_{t,c}(\tau)[i]$, for all irreducible $G$-representations $\tau$ and all possible grading shifts $i\in \mathbb{Z}$.
\end{prop}

\subsection{Characters}\label{char}

\begin{defn}\label{defchar}
Let $K_0(G)$ be the Grothendieck ring of the category of finite dimensional representations of $G$ over $\Bbbk$. For $M$ a representation of $G$, let $[M]$ be the corresponding element of $K_0(G)$. For $M=\oplus_{i}M_i$ any graded $H_{t,c}(G,\h)$ module with finite dimensional graded pieces, define its character to be 
$$\chi_{M}(z)=\sum_{i}[M_{i}] z^i \in K_0(G)[[z,z^{-1}]],$$ and its Hilbert series to be $$\mathrm{Hilb}_{M}(z)=\sum_{i}\dim(M_{i})z^i.$$
\end{defn} 

Any $M \in \mathcal{O}$ is finite dimensional and has character in $K_0(G)[z,z^{-1}]$.  Note that this notion of character differs from the usual class functions from the representation theory of finite groups.

\begin{prop}\label{charMLN}
\begin{enumerate}
\item The character of the Verma module $M_{t,c}(\tau)$ is
$$\chi_{M_{t,c}(\tau)}(z)=\sum_{i\ge 0}[S^{i}\h^* \otimes \tau]z^i,\,\,\,\,\,\,\,\, \mathrm{Hilb}_{M_{t,c}(\tau)}(z)=\frac{\dim(\tau)}{(1-z)^n}.$$

\item The character of the irreducible module $L_{t,c}(\tau)$ is
$$\chi_{L_{t,c}(\tau)}(z)=\sum_{i\ge 0}\left( [S^{i}\h^* \otimes \tau]-[\Ker B_i] \right) z^i,\,\,\,\,\,\,\,\, \mathrm{Hilb}_{M_{t,c}(\tau)}(z)=\sum_{i\ge 0} (\rank B_i) z^i .$$

\item The characters of baby Verma modules at $t=0$ and $t=1$ are related by 
$$\chi_{N_{1,c}(\tau)}(z)=\chi_{N_{0,c}(\tau)}(z^p)\left(\frac{1-z^p}{1-z} \right)^{\dim \h}.$$

\item If $G$ is such that the algebra of invariants $(S\h^*)^G$ is a polynomial algebra with homogeneous generators of degrees $d_1,\ldots d_n$, then the character and the Hilbert series of a baby Verma module $N_{t,c}(\tau)$, for $t=0,1$, is:
$$\chi_{N_{0,c}(\tau)}(z)=\chi_{M_{0,c}(\tau)}(z)(1-z^{d_1})(1-z^{d_2})\ldots (1-z^{d_n}).$$
$$\chi_{N_{1,c}(\tau)}(z)=\chi_{M_{1,c}(\tau)}(z)(1-z^{pd_1})(1-z^{pd_2})\ldots (1-z^{pd_n}).$$
$$ \mathrm{Hilb}_{N_{0,c}(\tau)}(z)=\dim(\tau)\frac{(1-z^{d_1})(1-z^{d_2})\ldots (1-z^{d_n})}{(1-z)^n}.$$
$$ \mathrm{Hilb}_{N_{1,c}(\tau)}(z)=\dim(\tau)\frac{(1-z^{pd_1})(1-z^{pd_2})\ldots (1-z^{pd_n})}{(1-z)^n}.$$

\end{enumerate}
\end{prop}

The matrix of the form $B_i$ on $M_{t,c}(\tau)_i$ depends polynomially on $c$. Combining this with the above proposition, we see that the character of $L_{t,c}(\tau)$ is the same for all generic $c$ (meaning, $c$ in a Zariski open set of $\mathbb{A}^n_\Bbbk$, where $n$ is the number of conjugacy classes of $G$), while for special $c$ (in the complimentary Zariski closed set), $\Ker B_i$ might be larger. In this paper, we will be interested only in characters at generic value of $c$.

\begin{prop}\label{generic}
\begin{enumerate}
\item 
For $t\ne 0$ and generic $c$, the character of $L_{t,c}(\tau)$ is of the form $$\chi_{L_{t,c}(\tau)}(z) =\chi_{S^{(p)}\h^*}(z) H(z^p),$$ where $S^{(p)}\h^*$ is the quotient of $S\h^*$ by the ideal generated by $x_1^p,\ldots x_n^p$, and $H\in K_{0}(G)[z]$ is a character of some finite dimensional graded $G$-representation. In particular, the Hilbert series of $L_{t,c}(\tau)$ is of the form 
$$\mathrm{Hilb}_{L_{t,c}(\tau)}(z) =\left( \frac{1-z^p}{1-z}\right)^n\cdot h(z^p),$$ for $h$ a polynomial with nonnegative integer coefficients. 
 \item The polynomial $h$ satisfies $1 \leq h(1) \leq |G|$, and $\dim L_{t,c}(\tau)=h(1) p^n$.
\end{enumerate} 
\end{prop}

\begin{defn} We call $H$ the \emph{reduced character}, and $h$ the \emph{reduced Hilbert series}  of $L_{t,c}(\tau)$.
\end{defn} 


The idea of the proof of \ref{generic} $(1)$ is to construct a subrepresentation $J_{t,0}'(\tau)$ of $M_{t,0}(\tau)$ which behaves as if $c=0$ were a generic point, even when it is not. More specifically, we pick a line through the origin in the space of all possible parameters $c$, such that all but finitely many points on this line are generic values of the parameter for the rational Cherednik algebra. 
We define $J_{t,0}'(\tau)$ to be the extension to $c=0$ of the map from a punctured $\Bbbk$-line to the appropriate Grassmanian associating to each generic $c$ the representation $J_{t,c}(\tau)$, and think of it as a limit of $J_{t,c}(\tau)$ as $c$ goes to zero.

This space $J_{t,0}'(\tau)$ is an $H_{t,0}(G,\h)$ subrepresentation, has the same Hilbert series as $J_{t,c}(\tau)$ for generic $c$, is contained in $J_{t,0}(\tau)$ (properly contained if $c=0$ is not a generic point), is $G$-invariant, and every graded piece $J_{t,0}'(\tau)_i$ has the same composition series as $J_{t,c}(\tau)_i$ for generic $c$.
As a consequence, the character of $L_{t,c}(\tau)$ at generic $c$ is the same as the character of $M_{t,0}(\tau)/J_{t,0}'(\tau)$. On the other hand, since $J_{t,0}'(\tau)$ is stable under Dunkl operators, 
it follows that $J_{t,0}'(\tau)$ is generated by $p$-th powers, i.e. elements of the form $f^p\otimes v$ for some $f\in S\h^*$, $v\in \tau$ (see Lemma 3.3 in \cite{BalChen}).



\begin{cor}\label{newp} Let $t\ne 0$ and $c$ be generic. 
The module $J_{t,c}(\tau)$ is generated under $S\h^*$ by homogeneous elements in degrees divisible by $p$. The images of such elements of degree $mp$ in the quotient $$(J_{t,c}(\tau)/\h^*J_{t,c}(\tau))_{mp}=J_{t,c}(\tau)_{mp}/\h^*J_{t,c}(\tau)_{mp-1}\subseteq S^{mp}\h^*\otimes \tau/\h^*J_{t,c}(\tau)_{mp-1}$$ form a subrepresentation of $S^{mp}\h^*\otimes \tau/\h^*J_{t,c}(\tau)_{mp-1}$ whose composition factors are a submultiset of composition factors of $(S^m\h^*)^p\otimes \tau/ (\h^*J_{t,0}'(\tau)_{mp-1}\cap (S^m\h^*)^p\otimes \tau)$.

Any such generator in degree $mp$ is a singular vector in the quotient of $M_{t,c}(\tau)$ by the $S\h^*$-submodule generated by all such generators from smaller degrees.
\end{cor}
\begin{proof} For representations $\sigma$ and $\sigma'$ of $G$, let us write $\sigma \preccurlyeq \sigma'$ if the multiset of composition factors of $\sigma$ is a subset of the multiset of  composition factors of $\sigma'$. If $\sigma$ and $\sigma'$ are graded $G$ representations, we write $\sigma \preccurlyeq \sigma'$ if $\sigma_i \preccurlyeq \sigma'_i$ for all $i$. If $\sigma \preccurlyeq \sigma'$ and $\sigma' \preccurlyeq \sigma$, then $[\sigma] = [\sigma']$ in the Grothendieck group. 

Let us first prove:  $$J_{t,c}(\tau)/\h^*J_{t,c}(\tau) \preccurlyeq J_{t,0}'(\tau)/\h^*J_{t,0}'(\tau) \preccurlyeq (S\h^*)^p\otimes \tau/ (\h^*J_{t,0}'(\tau)\cap (S\h^*)^p\otimes \tau ).$$  As $J_{t,c}(\tau)$ is a deformation of $J_{t,0}(\tau)'$, for any degree $i>0$ we have, in the Grothendieck group: $$[J_{t,0}(\tau)'_i]=[J_{t,c}(\tau)_i]$$ $$[J_{t,0}(\tau)'_{i-1}]=[J_{t,c}(\tau)_{i-1}]$$ $$\h^*J_{t,0}(\tau)'_{i-1} \preccurlyeq \h^*J_{t,c}(\tau)'_{i-1}$$ so $$J_{t,c}(\tau)/\h^*J_{t,c}(\tau) \preccurlyeq J_{t,0}'(\tau)/\h^*J_{t,0}'(\tau).$$ The statement $J_{t,0}'(\tau)/\h^*J_{t,0}'(\tau) \preccurlyeq  (S\h^*)^p\otimes \tau/ (\h^*J_{t,0}'(\tau)\cap (S\h^*)^p\otimes \tau )$ follows from $J_{t,0}'(\tau)$ being generated under $S\h^*$ by $p$-th powers. 

The module $J_{t,c}(\tau)$ is generated under $S\h^*$ by elements which have nonzero projection to $J_{t,c}(\tau)/\h^*J_{t,c}(\tau)$. Because of the above sequence of $\preccurlyeq$,  such elements only exist in degrees divisible by $p$, and their images in $J_{t,c}(\tau)/\h^*J_{t,c}(\tau) \subset S\h^*\otimes \tau/\h^*J_{t,c}(\tau)$ form a group representation which is $\preccurlyeq (S\h^*)^p\otimes \tau/ (\h^*J_{t,0}'(\tau)\cap (S\h^*)^p\otimes \tau ).$

For every $v\in J_{t,c}(\tau)_{mp}$ and every $y\in \h$, $D_{y}(v)\in J_{t,c}(\tau)_{mp-1}$. So, if $v$ is not in $\h^*J_{t,c}(\tau)_{mp-1}$, then its projection is a nonzero vector in $J_{t,c}(\tau)/\h^*J_{t,c}(\tau)$ with the property that $D_{y}(v)$ is zero in $J_{t,c}(\tau)/\h^*J_{t,c}(\tau)$, in other words a singular vector. 
\end{proof}

The reduced character $H(z)$ is computed as follows: if $J_{1,0}'(\tau)$ is generated by some collection of $p$-th powers $f_i(x_1^p,\ldots ,x_n^p)\otimes v_i$, then $H(z)$ is the character of the reduced module $$R_{t,c}(\tau)=S\h^*\otimes \tau /\left<f_i(x_1,\ldots x_n)\otimes v_i \right>.$$ The reduced module $R_{t,c}(\tau)$ is a $\Bbbk[G]\ltimes S\h^*$-module, but does not  have to be stable under Dunkl operators and might not be an $H_{t,c}(G,\h)$-module.

For $G=GL_{2}(\F_p)$, the character and Hilbert series of $S^{(p)}\h^*$ are
$$\chi_{S^{(p)}\h^*}(z)= \chi_{S\h^*}(z)-2\chi_{S\h^*}(z)z^p+\chi_{S\h^*}(z)z^{2p},\,\,\,\,\,\,\,\, \mathrm{Hilb}_{S^{(p)}\h^*}(z)=\left( \frac{1-z^p}{1-z}\right)^2.$$

\section{The group $GL_2(\F_p)$}

For the rest of the paper, we will be considering the rational Cherednik algebra associated to the group $GL_{2}(\F_p)$, for $\F_p\subseteq \Bbbk$ the finite field of $p$ elements. This group has $(p^2-1)(p^2-p)$ elements; since $p$ divides the order of the group, its category of representations is not semisimple. Let $\h_{\F}=\F_p^2$ and $\h=\Bbbk^2$ be the column vector representation of $G$, $y_1,y_2$ the tautological basis of $\h$ and $x_1,x_2$ the dual basis of $\h^*$. For a group element $g\in GL_{2}(\F_p)$, the matrix of its action on $\h$, written in $y_1,y_2$, is $g$, and the matrix of its action on $\h^*$, written in $x_1,x_2$, is $(g^{t})^{-1}$.
Let $$ d_{\lambda}=\left[\begin{array}{cc} \lambda ^{-1} & 0 \\ 0 & 1 \end{array}\right], \lambda\ne 1,  \,\,\,\,\, d_{1}=\left[\begin{array}{cc} 1 & 1 \\ 0 & 1 \end{array}\right].$$ 
be representatives of the conjugacy classes of reflections in $GL_{2}(\F_p)$ $C_{\lambda}=\{ gd_{\lambda}g^{-1} | g\in GL_{2}(\F_p) \}, \lambda \in \F_{p}^{\times}.$ For $\lambda \ne 1$, $C_{\lambda}$ contains $(p+1)p$ elements which we call \emph{semisimple reflections}; for $\lambda = 1$, $C_{\lambda}$ contains $(p+1)(p-1)$ elements which we call  \emph{unipotent reflections}. These sets of reflections can be parametrized as follows.

\begin{prop}\label{reflequiv} 
\begin{enumerate}
\item  There exists a bijection between the set of reflections and the set $\{ \alpha \otimes \alpha^\vee \in \h_{\F}^* \otimes \h_{\F} : (\alpha, \alpha^\vee) \ne 1 \}$. The reflection $s$ corresponding to $\alpha \otimes \alpha^\vee$ acts:
\begin{align*}
\text{on $\h^*$ by} \quad & s \ldot x = x - (\alpha^\vee, x) \alpha \\
\text{on $\h$ by} \quad & s \ldot y = y + \frac{(y, \alpha)}{1 - (\alpha, \alpha^\vee)} \alpha^\vee
\end{align*}
and belongs to $C_{\lambda}$ for $\lambda=1 - (\alpha^\vee, \alpha)$.
\item Under this bijection, the conjugacy classes of $GL_{2}(\F _p)$ are:
$$\lambda\ne 1:\,\,\,\, C_{\lambda}\leftrightarrow \{\left[\begin{array}{c} 1 \\ b \end{array} \right] \otimes \left[\begin{array}{c} 1-\lambda -bd \\ d \end{array} \right]  | b,d\in \F_p  \} \cup \{ \left[\begin{array}{c} 0 \\ 1 \end{array} \right] \otimes \left[\begin{array}{c} a \\ 1-\lambda \end{array} \right]  | a \in \F_p \} $$
$$ \,\,\,\, \,\,\,\,C_{1} \leftrightarrow\{\left[\begin{array}{c} 1 \\ b \end{array} \right] \otimes \left[\begin{array}{c} -bd \\ d \end{array} \right]  | b,d\in \F_p, d\ne 0  \} \cup \{ \left[\begin{array}{c} 0 \\ 1 \end{array} \right] \otimes \left[\begin{array}{c} a \\ 0 \end{array} \right]  | a \in \F_p, a\ne 0 \}.$$
\end{enumerate}
\end{prop}

\subsection{Invariants and characters of baby Verma modules}\label{babychar}

Proposition \ref{charMLN} states that the characters of baby Verma modules are easy to compute for groups $G$ for which the algebra of invariants $(S\h^*)^G$ is a polynomial algebra. Dickson \cite{dickson} showed that all $GL_{n}(\F_{p^r})$ are such groups, and constructed the invariants explicitly. We recall the construction for $GL_{2}(\F_p)$.

For $(n,m)$ an ordered pair of nonnegative integers, define $[n,m] \in S\h^*$ by
$$[n,m]=\det \left[ \begin{array}{cc} x_1^{p^n} & x_2^{p^n} \\ x_1^{p^m} & x_2^{p^m} \end{array} \right].$$ When acting on $S\h^*$ by $g\in GL_{2}(\F_p)$, this element transforms as 
$g\ldot [n,m]=(\det g)^{-1} [n,m].$ From this it follows that for $L=[1, 0]$, 
$$Q_{0}=L^{p-1}=[1,0 ]^{p-1}=(x_1^px_2-x_1x_2^p)^{p-1}=\sum_{i=0}^{p-1}x_1^{(p-1)(p-i)}x_2^{(p-1)(i+1)}$$
$$Q_{1}=\frac{[2,0]}{L}=\frac{[2,0]}{[1,0]}=\frac{x_1^{p^2}x_2-x_1x_2^{p^2}}{x_1^px_2-x_1x_2^p}=\sum_{i=0}^px_1^{(p-1)(p-i)}x_{2}^{(p-1)i}$$
are invariants. The main theorem of  \cite{dickson}, applied to $GL_{2}(\F_p)$, yields:
\begin{prop}
The ring of invariants $(S\h^*)^{GL_{2}(\F_p)}$ is a polynomial ring in $Q_0$ and $Q_1$.
\end{prop}

Finally, as $\deg Q_0=p^2-1$ and $\deg Q_1=p^2-p$, we find the character formulas for baby Verma modules for $GL_2(\F_p)$:
$$\chi_{N_{0,c}(\tau)}(z)=\chi_{M_{0,c}(\tau)}(z)(1-z^{(p^2-1)})(1-z^{(p^2-p)}),$$
$$ \mathrm{Hilb}_{N_{0,c}(\tau)}(z)=\dim(\tau)\frac{(1-z^{(p^2-1)})(1-z^{(p^2-p)})}{(1-z)^2}.$$

$$\chi_{N_{1,c}(\tau)}(z)=\chi_{M_{t,c}(\tau)}(z)(1-z^{p(p^2-1)})(1-z^{p(p^2-p)}),$$
$$ \mathrm{Hilb}_{N_{1,c}(\tau)}(z)=\dim(\tau)\frac{(1-z^{p(p^2-1)})(1-z^{p(p^2-p)})}{(1-z)^2}.$$

\subsection{Representations of $GL_2(\F_p)$}

The following fact can be found in \cite{brauer}:

\begin{prop}All irreducible representations of  $GL_{2}(\mathbb{F}_{p})$ over $\Bbbk$ are of the form $$S^{i}\h \otimes (det)^j,$$ for $i=0,1,\ldots p-1$, $j=0,\ldots p-2.$ 
\end{prop}

\begin{lemma}\label{duals}
As representations of $GL_{2}(\mathbb{F}_{p})$, $\h^*\cong \mathfrak{h}\otimes (det)^{-1},$ and the isomorphism is $x_{1}\mapsto -y_2,\, x_2\mapsto y_1$.
\end{lemma}
\begin{proof}
This follows directly from the nondegenerate pairing $\h \otimes \h \rightarrow \wedge^2 \h \cong \det$.
\end{proof}

Next, we will need to know how the higher symmetric powers and tensor products of such representations decompose into irreducible components. The following two propositions are well known, and can be proved directly or deduced from \cite{glover1}.
\begin{prop}\label{tensorprod}
For any $i,j>0$, there is a short exact sequence of $GL_2(\F_p)$-representations
$$0 \rightarrow S^{i-1}\h\otimes S^{j-1}\h\otimes \det \rightarrow S^{i}\h\otimes S^{j}\h \rightarrow S^{i+j}\h \rightarrow 0.$$
The first map is given by $f\otimes g \mapsto (y_{1}\otimes y_{2}-y_{2}\otimes y_{1}) \cdot f \otimes g$, and the second map is $f\otimes g \rightarrow f\cdot g.$ 
\end{prop}

\begin{prop}\label{reducibles}
Let $0 \leq j < p,$ and $n\ge 0$. There is a short exact sequence of $GL_2(\F_{p})$-representations:
$$0 \rightarrow S^j \h \otimes S^n \h \rightarrow S^{j + pn} \h \rightarrow S^{p - j - 2}\h \otimes S^{n-1}\h \otimes \mathrm{det}^{j+1} \rightarrow 0.$$
Here we use the convention $S^{i}\h=0$ if $i<0$. The first map is
$$y_1^{a}y_2^b \otimes y_1^{c}y_2^d\mapsto y_1^{a+cp}y_2^{b+dp}$$
and the second, for $0\le a,b<p$, is
$$ y_1^{a+pc}y_2^{b+pd}\mapsto \left\{ \begin{array}{r@{\, ,\ }l} {a \choose j+1}\cdot y_1^{a-j-1}y_2^{b-j-1}\otimes y_1^{c}y_2^{d}  & a+b=p+j \\ 0 & otherwise.   \end{array} \right.$$
\end{prop}

\section{Category $\mathcal{O}$ for the rational Cherednik algebra $H_{t,c}(GL_2(\F_p),\h)$}

The main theorem of the paper is the following:

\begin{thm}\label{mainII}
Up to a grading shift, any irreducible representations in category $\mathcal{O}$ for the rational Cherednik algebra $H_{t,c}(GL_2(\F_p),\h)$ is isomorphic to $L_{t,c}(S^{i}\h\otimes (\det)^{j})$ for some $0\le i \le p-1$ and $0\le j\le p-2$. The characters and Hilbert polynomials of these representations are as follows. 

For $t=0$ and generic $c$:

\vspace{0.5cm}
\begin{tabular}{|c|c|c|}
\hline$i$ & $\chi_{L_{0,c}(S^i\h \otimes \det^j)}(z)$ & $\mathrm{Hilb}_{L_{0,c}(S^i\h \otimes \det^j)}(z)$ \\
\hline $[0, p-3]$ & $\displaystyle [S^i\h \otimes (\det)^{j}]$ & $\displaystyle i+1$ \\
\hline$p-2$ & $\displaystyle [S^{p-2}\h \otimes (\det)^{j}]+[S^{p-1}\h\otimes (\det)^{j-1}]z$ & $\displaystyle (p-1)+pz+(p-1)z^2.$  \\
& $+[S^{p-2}\h \otimes (\det)^{j-1}]z^2$ & \\
\hline$p-1$ & $\displaystyle \chi_{M_{0,c}(S^{p-1}\h \otimes (\det)^{j})}(z)(1-z^{p-1})(1-z^{p^2-1})$ & $\displaystyle p\frac{(1-z^{p-1})(1-z^{p^2-1})}{(1-z)^2}$\\
\hline
\end{tabular}
\vspace{0.5cm}

For $t=1$, generic $c$ and any $\tau$, the reduced character and the reduced Hilbert polynomial of $L_{1,c}(\tau)$ are $$H_{L_{1,c}(\tau)}=\chi_{L_{0,c}(\tau)},\,\,\,\,\,\, h_{L_{1,c}(\tau)}=\mathrm{Hilb}_{L_{0,c}(\tau)}.$$
The character and Hilbert polynomial of $L_{1,c}(\tau)$ are
$$\chi_{L_{1,c}(\tau)}(z) =H_{L_{1,c}(\tau)}(z^p)\chi_{S^{}\h^*}(z)(1-z^p)^2,\,\,\,\,\mathrm{Hilb}_{L_{1,c}(\tau)}(z) =h_{L_{1,c}(\tau)}(z^p)\left( \frac{1-z^p}{1-z}\right)^2 .$$ 

\end{thm}
\begin{proof}
Lemma \ref{det} shows how all the formulas for the characters of $L_{t,c}(S^i\h\otimes (\det)^j)$ follow from the ones for $j=0$. Those are proved for $0\le i<p-2$ in Propositions \ref{0,smalli} (for $t=0$) and \ref{1,smalli} (for $t=1$); for $i=p-2$ in Propositions \ref{0,p-2} (for $t=0$) and \ref{1,p-2} (for $t=1$), and for $i=p-1$ in Propositions  \ref{charm} and \ref{0,p-1,final} (for $t=0$) Propositions \ref{1charm} and \ref{1,p-1,final} (for $t=1$).
    \end{proof}

\begin{rem}
For $G=GL_2(\F_p)$, the reduced character $H_{L_{1,c}(\tau)}$ for $L_{1,c}(\tau)$ is equal to the character $\chi_{L_{0,c}(\tau)}$ of $L_{0,c}(\tau)$ for all $\tau$. The analogous statement is always true for baby Verma modules and one might be tempted to conjecture that it holds for irreducible modules in general. This is however not true: a counterexample is $G=SL_2(\F_3)$, $\tau=\triv$.
\end{rem}

\subsection{Blocks}

Consider the following element of $H_{t,c}(GL_{2}(\F_p),\h)$ (see \cite{etingof-ma}):
$$\mathbf{h}=\sum_{i=1,2} x_iy_i-\sum_{s\in S}c_s s.$$ 
Direct computation shows that $$[\mathbf{h},x] =tx,\, [\mathbf{h},y] = -ty,\, [\mathbf{h},g] =0$$ for $x \in \mathfrak{h}^*, y \in \mathfrak{h}, g\in G$. 

For every conjugacy class $C$ of reflections, $\sum_{s\in C} s $ is central in the group algebra of $GL_{2}(\F_p)$, so it acts as a constant on every irreducible representation. Consider the lowest weight subspace $\tau \cong M_{t,c}(\tau)_{0}\subseteq M_{t,c}(\tau)$: $y_i$ act on it as $0$, so $\mathbf{h}$ acts on it by a constant $h_c(\tau)=-\sum_{s\in S}c_s s|_{\tau}.$ From this and $[\mathbf{h},x] =tx$ it follows that $\mathbf{h}$ acts on $M_{t,c}(\tau)_m$ by $h_c(\tau)+tm\in \Bbbk$.

Now, assume that $L_{t,c}(\sigma)[m]$ is a composition factor of $M_{t,c}(\tau)$. Then $\mathbf{h}$  acts on $\sigma \subseteq M_{t,c}(\tau)_m$ as $h_c(\sigma)$ and as $h_c(\tau)+tm$, so $$h_c(\sigma)=h_c(\tau)+tm.$$

If $t=0$, the above formula simplifies to $h_c(\sigma)=h_c(\tau)$. If $t=1$ and $c$ is generic, then Proposition \ref{generic} implies that  the only composition factors in $M_{1,c}(\tau)$ are of the form $L_{1,c}(\sigma)[mp]$, so the above formula also simplifies to $h_c(\sigma)=h_c(\tau)$. Thus, we get the following lemma. 

\begin{lemma}\label{hc1}
Let $t=0$ or let $t=1$ and $c$ be generic. If $L_{t,c}(\sigma)[m]$ is a composition factor of $M_{t,c}(\tau)$, then $h_c(\sigma)=h_c(\tau)$.
\end{lemma}

The constants $h_c(\tau)$ are easy to calculate directly. 

\begin{lemma} For $GL_{2}(\F_p)$ and conjugacy class $C_{\lambda}$ of reflections in it, the action of the central element $\sum_{s\in C_{\lambda}}s$  on symmetric powers of the reflection representation is:
$$\textrm{For } \lambda\ne 1,\,\,\,\, \sum_{s\in C_{\lambda}}s |_{S^{i}\h}=\left\{ \begin{array}{r@{\, ,\ }l} 0 & i< p-1 \\ 1 & i=p-1   \end{array} \right.$$
$$\textrm{For } \lambda= 1,\,\,\,\, \sum_{s\in C_{1}} s |_{S^{i}\h}=\left\{ \begin{array}{r@{\, ,\ }l} -1 & i< p-1 \\ 0 & i=p-1   \end{array} \right.$$
So, for $\tau=S^{i}\h \otimes \mathrm{det}^j$, the action of $\mathbf{h}$ on the lowest weight $\tau \subseteq M_{t,c}(\tau)$ is by the constant $$h_c(\tau)= \left\{ \begin{array}{r@{\, ,\ }l} c_1 & i< p-1 \\ -\sum_{\lambda \ne 0,1} \lambda^j c_{\lambda} & i=p-1   \end{array} \right.$$
\end{lemma}
\begin{proof}

We use the parametrization of conjugacy classes from Proposition \ref{reflequiv}. As $\sum_{s\in C_{\lambda}}s$ is central in the group algebra and acts on $S^{i}\h$ as a constant, it is enough to compute $\sum_{s\in C_{\lambda}}s.y_1^i$. We may disregard all terms of the form $y_1^{i-j}y_2^j$ for $j>0$, as we know these sum up to zero. We use Lemma 2.26 of \cite{BalChen} several times.

For $\lambda \ne 1$, the action of $\sum_{s\in C_{\lambda}}s$ on $y_1^i$ is by a constant: 
$$\sum_{b,d\in \F_{p}} \left( 1+\frac{1}{\lambda}\cdot 1\cdot (1-\lambda-bd)\right)^i + \sum_{a\in \F_p}\left(1+\frac{1}{\lambda}\cdot 0 \right) ^i=$$
$$=\frac{1}{\lambda^i} \sum_{b,d\in \F_{p}} \left( 1-bd\right)^i =\frac{-1}{\lambda^i}\sum_{m\in \F_p} m^i = -\sum_{m\in \F_p} m^i = \left\{ \begin{array}{r@{\, ,\ }l} 0 & i< p-1 \\ 1 & i=p-1   \end{array} \right.$$

For $\lambda=1$, a similar computation yields:
$$\sum_{\substack{b,d\in \F_{p} \\ d\ne 0}} \left( 1-bd\right)^i + \sum_{\substack{a\in \F_p\\ a\ne 0}}1^i=(p-1)\left( \sum_{m\in \F_{p}}m^i+1 \right)=\left\{ \begin{array}{r@{\, ,\ }l} -1 & i< p-1 \\ 0 & i=p-1.  \end{array} \right. $$

The action of $\mathbf{h}$ on irreducible $S^i\h \otimes \det^j$ can now be computed directly.
\end{proof}

\begin{cor}\label{1block} For generic $c$, the representations of the form $L_{t,c}(S^{p-1}\h\otimes \det ^j)$ form blocks of size one, meaning that the only irreducible representations that appear as composition factors in any representation with lowest weight $S^{p-1}\h\otimes \det ^j$ are isomorphic, up to grading shifts, to $L_{t,c}(S^{p-1}\h\otimes \det ^j)$. \label{blocks}
\end{cor}

\subsection{Dependence of the character of $L_{t,c}(S^{i}\h\otimes (\det)^j)$ on $j$}

We start the character computations with a reduction which allows us to only consider the case $j=0$.

For any reflection group $G$, its reflection representation $\h$, a character $\psi$ of $G$, there is an isomorphism of algebras  $H_{t,c}(GL_2(\F_p),\h)\cong H_{t,c\cdot \psi}(GL_2(\F_p),\h)$ defined to be the identity on $\h$ and $\h^*$, and sending $g\in G$ to $(\psi (g))\cdot g$. Twisting by this isomorphism makes a representation $L_{t,c\cdot \psi}(\tau)$ of $H_{t,c\cdot \psi} (G,\h)$ into a representation $L_{t,c}(\tau\otimes \psi)$ of $H_{t,c} (G,\h)$. As the character (as defined above) of $L_{c}(\tau)$ is the same for all generic values of $c$, we find that $\chi_{L_{t,c}(\tau\otimes \psi)}=\chi_{L_{t,c}(\tau)}\cdot [\psi],$ and their Hilbert series are the same.

\begin{lemma}\label{det}
For the rational Cherednik algebra $H_{t,c}(GL_2(\F_p),\h)$ defined at the generic value of parameter $c$, the Hilbert polynomials and the characters of irreducible representations satisfy
$$\mathrm{Hilb}_{L_{t,c}(\tau\otimes \det)}=\mathrm{Hilb}_{L_{t,c}(\tau)}\cdot [\det].$$
$$\chi_{L_{t,c}(\tau\otimes \det)}=\chi_{L_{t,c}(\tau)}\cdot [\det].$$
\end{lemma}

\section{Characters of $L_{t,c}(S^{i}\h)$ for $i=0\ldots p-3$, $p \geq 3$}

\subsection{Characters of $L_{t,c}(S^{i}\h)$ for $i=0\ldots p-3$ and $t=0$}

\begin{prop}\label{0,smalli}
For $i=0,\ldots p-3$, $t=0$ and all $c$, the space $M_{0,c}(S^i\h)_1$ consists  of singular vectors. So, the character of $L_{0,c}(S^{i}\h)$ is
$$\chi_{L_{0,c}(S^{i}\h)}(z)=[S^i\h], \,\,\,\,\, \mathrm{Hilb}_{L_{0,c}(S^{i}\h)}(z)=i+1.$$
\end{prop}
\begin{proof}
The space $M_{0,c}(S^i\h)_1$ is isomorphic to $\h^*\otimes S^i\h$ as a $GL_{2}(\F_p)$- representation. To show that it consists of singular vectors, we will show that for any $x\in \h^*$, any $y\in h$, and any $f\in S^i\h$, 
$D_y(x \otimes f) = 0$.  Summing over each $C_\lambda$ and using Proposition \ref{reflequiv}, this is equivalent to showing that
$$\sum_{s\in C_{\lambda}}(y,\alpha_s)\frac{(1-s)\ldot x}{\alpha_s}\otimes s\ldot f = \sum_{\substack{\alpha\otimes \alpha^{\vee}\ne 0 \\(\alpha,\alpha^\vee)=1-\lambda}} (y,\alpha)(x,\alpha^\vee)\otimes s\ldot f = 0.$$

Using the parameterization in Proposition \ref{reflequiv} (2), 
write the above sum as a sum over all  $a,b,d\in \F_p$. Then $s\ldot f$ is a polynomial of degree $i$ in $a$ or  in $d$, so the summand $(y,\alpha)(x,\alpha^\vee)\otimes s\ldot f$ is a polynomial in degree $1+i\le p-2<p-1$ in $a$ or in $d$. By Lemma 2.26 of \cite{BalChen}, this means the sum is zero as claimed.
\end{proof}

\subsection{Characters of $L_{t,c}(S^{i}\h)$ for $i=0\ldots p-3$ and $t=1$}

A very similar computation gives the analogous answer in case $t=1$.

\begin{prop}\label{1,smalli}
For $i=0,\ldots p-3$, $t=1$ and all $c$, all the vectors of the form $x^p\otimes v \in S^p\h^*\otimes S^i\h\cong M_{1,c}(S^i\h)_p$ are singular. For generic $c$ these vectors generate $J_{1,c}(S^i\h)$, so the character and the Hilbert polynomial of the irreducible module $L_{1,c}(S^{i}\h)$ are
$$\chi_{L_{1,c}(S^{i}\h)}(z)=\chi_{S^{(p)}\h^*}(z)\cdot [S^i\h],\,\,\,\,\,\, \mathrm{Hilb}_{L_{1,c}(S^{i}\h)}(z)=(i+1) \left(\frac{1-z^p}{1-z} \right)^2, $$ 
and its reduced character and Hilbert polynomial are
$$H(z)=[S^i\h], \,\,\,\,\,\,\,\,\,\, h(z)=i+1.$$
\end{prop}
\begin{proof}
The proof is very similar to the proof of the previous proposition. To show that all vectors of the form $x^p\otimes v \in S^p\h^*\otimes S^i\h\cong M_{1,c}(S^i\h)_p$ are singular, we need to show that 
$$D_{y}(x^p\otimes f)=-\sum_{\lambda}c_\lambda \sum_{\alpha}(y,\alpha) \alpha^{p-1}\sum_{\alpha^\vee}(x,\alpha^{\vee})^p\otimes (s\ldot f) = 0.$$
Again use Proposition \ref{reflequiv} to write this as a sum over all $a,b,d\in \F_p$ parametrizing $\alpha\otimes \alpha^{\vee}$. We may assume that $x\in \h_{\F}^*\subseteq \h^*$ by picking a convenient basis; the claim will then be true for $\Bbbk$-linear combinations of such $x$ as well. For $x\in  \h_{\F}^*$, $(x,\alpha^{\vee})$ is in $\F_p\subseteq \Bbbk$, and $(x,\alpha^{\vee})^p=(x,\alpha^{\vee})$. Using this, the inner sum over $\alpha^{\vee}$ again becomes a sum over all $d\in \F_p$ or over all $a \in \F_p$ of a polynomial $(x,\alpha^{\vee})\otimes (s\ldot f)$ of degree $1+i<p-1$ in $d$ or $a$, so the sum is zero by Lemma 2.26 of \cite{BalChen} and the vectors $x^p\otimes v$ are singular.

To see that these vectors generate $J_{1,c}(S^i\h)$ at generic $c$, we use Proposition \ref{generic}, by which the character of $L_{1,c}(S^i\h)$ is of the form $$\chi_{L_{t,c}(\tau)}(z) =\chi_{S^{(p)}\h^*}(z) H(z^p).$$ The character of the quotient of $M_{1,c}(S^i\h)$ by the singular vectors found in this lemma is $$\chi_{L_{t,c}(\tau)}(z) =\chi_{S^{(p)}\h^*}(z) [S^i\h],$$ and so the graded $GL_2(\F_p)$-representation with the character $H(z)$ is a nonzero quotient of the irreducible representation $[S^i\h]$ concentrated in one degree. Thus, $H(z)=[S^i\h]$ and the character of $L_{1,c}(S^i\h)$ is as claimed.
\end{proof}



To get formulas relating the charactes of $L_{1,c}(S^i\h)$, $N_{1,c}(S^i\h)$ and $M_{1,c}(S^i\h)$, set $d_1=p^2-p$, $d_2=p^2-1$. Using $$\chi_{N_{1,c}(\tau)}=\chi_{M_{1,c}(\tau)}(1-z^{pd_1})(1-z^{pd_2}),$$ 
we have
$$L_{1,c}(S^{i}\mathfrak{h})=M_{1,c}(S^{i}\mathfrak{h})-M_{1,c}( \mathfrak{h}^*\otimes S^{i}\mathfrak{h})[p]+M_{1,c}( \det  \otimes S^{i}\mathfrak{h})[2p]$$ and 
$$\chi_{L_{1,c}(S^i\h)}=\frac{ \chi_{N_{1,c}(S^{i}\mathfrak{h})}(z)-\chi_{N_{1,c}( \mathfrak{h}^*\otimes S^{i}\mathfrak{h})}(z)z^p+\chi_{N_{1,c}( \det  \otimes S^{i}\mathfrak{h})}(z)z^{2p}} {(1-z^{pd_1})(1-z^{pd_2})}.$$

\section{Characters of $L_{t,c}(S^{i}\h)$ for $i= p-2$}

\subsection{Characters of $L_{t,c}(S^{i}\h)$ for $i=p-2$ and $t=0$}

\begin{prop}\label{0,p-2}
The character and and the Hilbert polynomial of
 $L_{0,c}(S^{p-2}\h)$ are
$$\chi_{L_{1,c}(S^{p-2}\h)}(z) = [S^{p-2}\h]+[S^{p-1}\h\otimes (\det)^{-1}]z+[S^{p-2}\h \otimes (\det)^{-1}]z^2,$$
$$\mathrm{Hilb}_{L_{1,c}(S^{p-2}\h)}(z) = (p-1)+pz+(p-1)z^2.$$
\end{prop}
\begin{proof}
We will prove this in a series of lemmas. Let us outline the proof here, and define several auxiliary modules, used only in this subsection. 

The character of the Verma module $M_{0,c}(S^{p-2}\h)$ is $$\chi_{M_{0,c}(S^{p-2}\h)}(z)=\sum_{j\ge 0}[S^j\h^*\otimes S^{p-2}\h]z^j.$$ Lemma \ref{0,p-2,1} shows that the space of singular vectors in $M_{0,c}(S^{p-2}\h)_1$ is isomorphic to $S^{p-3}\h$, and consequently that $J_{0,c}(S^{p-2}\h)_1\cong S^{p-3}\h$. We define $M^{1}$ to be the quotient of the Verma module $M_{0,c}(S^{p-2}\h)$ by the submodule generated by these vectors. 

The character of $M^{1}$ begins as
$$\chi_{M^{1}}(z)=[S^{p-2}\h]+([\h^*\otimes S^{p-2}\h]-[S^{p-3}\h])z+([S^2\h^*\otimes S^{p-2}\h]-[\h^*\otimes S^{p-3}\h])z^2+$$ $$+([S^3\h^*\otimes S^{p-2}\h]-[S^2\h^*\otimes S^{p-3}\h])z^3+\ldots,$$ which is, using Lemmas \ref{tensorprod} and \ref{reducibles}, equal to 
$$\chi_{M^{1}}(z)=[S^{p-2}\h]+[S^{p-1}\h\otimes (\det)^{-1}]z+[S^{p}\h \otimes (\det)^{-2}]z^2+[S^{p+1}\h \otimes (\det)^{-3}]z^3+\ldots.$$
The module $M^1$ has the property that its zeroth and first graded piece are equal to those of the irreducible module $L_{0,c}(\tau)$. 

However, $M^1$ is not irreducible. Lemma \ref{0,p-2,2} shows that the space of singular vectors in
$M^1_2\cong S^{p}\h \otimes (\det)^{-2}$ is isomorphic to $\h\otimes (\det)^{-2}$. This subspace is thus also in $J_{0,c}(S^{p-2}\h)$. Define $M^2$ as the quotient of $M^1$ by the submodule generated by these vectors. The module $M^2$ is equal to $L_{0,c}(S^{p-2}\h)$ in graded pieces $0,1$ and $2$, and its character begins as
$$\chi_{M^{2}}(z)=[S^{p-2}\h]+[S^{p-1}\h\otimes (\det)^{-1}]z+([S^{p}\h \otimes (\det)^{-2}]-[\h\otimes (\det)^{-2}])z^2+$$ $$+([S^{p+1}\h \otimes (\det)^{-3}]-[\h^*\otimes \h\otimes (\det)^{-2}])z^3+\ldots=$$
$$=[S^{p-2}\h]+[S^{p-1}\h\otimes (\det)^{-1}]z+[S^{p-2}\h \otimes (\det)^{-1}]z^2+([S^{p-3}\h \otimes (\det)^{-1}]z^3+\ldots$$

Finally, Lemma \ref{0,p-2,3} shows that $M^2_3\cong S^{p-3}\h \otimes (\det)^{-1}$ is consists of singular vectors. From this it follows that the quotient of $M^2$ by this subspace, is irreducible and equal to $L_{0,c}(S^{p-2}\h)$.

This proves the proposition, modulo Lemmas \ref{0,p-2,1}, \ref{0,p-2,2} and \ref{0,p-2,3}.

\end{proof}

\begin{lemma}\label{0,p-2,1}
The space of singular vectors in $M_{0,c}(S^{p-2}\h)_1\cong \h^*\otimes S^{p-2}\h$ is isomorphic to $S^{p-3}\h$ and consists of all vectors of the form $x_1\otimes y_1f+x_2\otimes y_2f, \,\,\, f\in S^{p-3}\h.$
\end{lemma}
\begin{proof}
As a $GL_{2}(\F_p)$-representation, the first graded piece of the Verma module, $M_{0,c}(S^{p-2}\h)_1$ is isomorphic to $\h^*\otimes S^{p-2}\h$. By Lemma \ref{duals}, this is isomorphic to $\h\otimes S^{p-2}\h\otimes (\det)^{-1}$, and by Lemma \ref{tensorprod}, it fits into a short exact sequence
$$0 \rightarrow S^{p-3}\h \rightarrow \h\otimes S^{p-2}\h\otimes (\det)^{-1} \rightarrow S^{p-1}\h \otimes (\det)^{-1} \h\rightarrow 0.$$
The irreducible subrepresentation isomorphic to $S^{p-3}\h$ includes into $\h^*\otimes S^{p-2}\h$ by $f\mapsto x_1\otimes y_1f+x_2\otimes y_2f$. Both this subrepresentation and the quotient are irreducible. 

If a vector $v \in M_{0,c}(S^{p-2}\h)_1$ is contained in the kernel of $B$, which is the maximal proper graded submodule $J_{0,c}(S^{p-2}\h)$, then action on it by $y\in \h$ produces an element of $J_{0,c}(S^{p-2}\h)_0$. However, the form is nondegenerate in degree $0$, and  $L_{0,c}(S^{p-2}\h)_0= M_{0,c}(S^{p-2}\h)_0$, so $y\ldot v=0$. In other words, such a vector is singular. 

To show that $J_{0,c}(S^{p-2}\h)_1\cong S^{p-3}\h$, we are going to show that: 
\begin{enumerate}
\item At least one nonzero vector from $S^{p-3}\h$ is singular;
\item Not all vectors in $M_{0,c}(S^{p-2}\h)_1$ are singular.
\end{enumerate}
The space of singular vectors is invariant under the group action, and both $S^{p-3}\h$ and the quotient are irreducible, so this proves the claim.

First, to show that the space $S^{p-3}\h$ consist of singular vectors, it is enough to show that $x_1\otimes y_1f+x_2\otimes y_2f$ is singular for some $f\in S^{p-3}\h$. Pick $f$ symmetric with respect to changing indices $1$ and $2$, and then it is enough to show that  $D_{y_1}(x_1\otimes y_1f+x_2\otimes y_2f) = 0.$ 
We use the parametrization of conjugacy classes from Proposition \ref{reflequiv} and the definition of Dunkl operator $D_{y_1}$, and denote $\alpha_b=x_1+bx_2$, and see that the coefficient of $-c_{\lambda}$ in $D_{y_1}(x_1\otimes y_1f+x_2\otimes y_2f)$ is
\begin{eqnarray*}
&& \sum_{\alpha\otimes \alpha^{\vee} \in C_{\lambda}}(y_1,\alpha)\frac{1}{\alpha}\left( (x_1-s.x_1)\otimes (s.y_1) (s.f) +(x_2-s.x_2)\otimes (s.y_2) (s.f)  \right) \\
&=& \sum_{b,d} 1\cdot \frac{1}{\alpha_b}\left( (1-\lambda-bd)\alpha_b \otimes (s.y_1) (s.f) + d\alpha_b\otimes (s.y_2) (s.f) \right)\\
&=& 1\otimes \sum_{b,d} \left( (1-\lambda-bd) s.y_1  + d s.y_2 \right) s.f\\
&=& 1\otimes \sum_{b,d}\frac{1}{\lambda} \left( (1-\lambda-bd)y_1  + d y_2 \right) s.f.\\
\end{eqnarray*}
The sum is over all $b,d\in \F_p$ if $\lambda \ne 1$, and over all $b,d\in \F_p$ with $d\ne 0\in \F_p$  if $\lambda=1$. However, if $\lambda=1$, then the $d=0$ term does not contribute to the sum, so let us consider the sum to be over all $b,d\in \F_{p}$ in both cases. The term $s.f$ is a vector in $S^{p-3}\h$ with coefficients polynomials in $b,d$ whose degree in $b$ and in $d$ is less or equal to $p-3$. The overall expression is a sum over all $b,d\in \F_p$ of polynomials whose degree in each variable is $\le p-2$, and it is thus zero by Lemma \ref{ntuplesum}.

So, the subspace isomorphic to $S^{p-3}\h$ indeed consists of singular vectors. 

To see that the space of singular vectors in $M_{0,c}(S^{p-2}\h)_1$ is not the whole space, it is enough to find one vector which is not singular. For example, one can compute that $D_{y_1}(x_1\otimes (y_1)^{p-2})$ has a coefficient of $c_{1}y_1^{p-2}$ equal to $-1$, so $x_1\otimes (y_1)^{p-2}$ is not singular.
\end{proof}

In the proof of Proposition \ref{0,p-2} we defined $M^1$ as the quotient of $M_{0,c}(S^{p-2}\h)$ by the submodule generated by hte singular vectors $x_1\otimes y_1f+x_2\otimes y_2f,$  from the previous lemma. The modules $M^1$ agrees with $L_{0,c}(S^{p-2}\h)$ in graded pieces $0$ and $1$, and that its character is 
$$\chi_{M^{1}}(z)=[S^{p-2}\h]+[S^{p-1}\h\otimes (\det)^{-1}]z+[S^{p}\h \otimes (\det)^{-2}]z^2+[S^{p+1}\h \otimes (\det)^{-3}]z^3+\ldots.$$ Next, we look for singular vectors in $M^1_2$.

\begin{lemma}\label{0,p-2,2}
The space of singular vectors in $M^1_2$ is isomorphic to $ \h \otimes (\det)^{-2} $. The representatives of these vectors in the Verma module $M_{0,c}(S^{p-2}\h)$ are linear combinations of $x_2^2\otimes y_1^{p-2}$ and $x_1^2\otimes y_2^{p-2}$.
\end{lemma}

\begin{proof}
The space $M^1_2\cong S^{p}\h\otimes (\det)^{-2}$ fits, by Lemma \ref{reducibles}, into a short exact sequence of $GL_2(\F_p)$- representations
$$0 \rightarrow \h \otimes (\det)^{-2} \rightarrow S^{p} \h \otimes (\det)^{-2} \rightarrow S^{p - 2}\h \otimes \det  \rightarrow 0.$$
We claim that the irreducible subrepresentation consists of singular vectors, but that the quotient is not in the kernel of $B$. Tracking through all the inclusions, quotient maps and isomorphisms in the previous lemma shows that $x_2^2\otimes y_1^{p-2}$ and $x_1^2\otimes y_2^{p-2} \in M_{0,c}(S^{p-2}\h)_2$ really map to the basis of $\h \otimes (\det)^{-2}$ under the quotient map $M_{0,c}(S^{p-2}\h)\to M^1$.

We use the following observation. For any rational Cherednik algebra module $N$, and any $y\in \h$, $n\in N$, $g\in G$, the relations of the rational Cherednik algebra imply that 
$g\ldot(y\ldot n)= (g\ldot y)\ldot (g\ldot n),$ so the map $\h\otimes N_{i}\to N_{i-1}$ given by $y\otimes n\mapsto y\ldot n$ is a map of $GL_2(\F_p)$-representations. 

In particular, applying the Dunkl operator is a homomorphism $\h\otimes  M^1_2 \to M^1_1.$ Showing that $\h\otimes (\det)^{-2}\subseteq M^1_2$ consists of singular vectors  is equivalent to showing that the restriction of the above map to this space $\h\otimes \h\otimes (\det)^{-2} \to M^1_1$ is zero. To do this, notice that the short exact sequence calculating the composition series of $\h\otimes \h\otimes (\det)^{-2}$ is
$$0 \rightarrow (\det)^{-1} \rightarrow \h\otimes \h\otimes (\det)^{-2} \rightarrow S^{2}\h \otimes (\det)^{-2} \rightarrow 0,$$ while the target space of the homomorphism is the irreducible $M^1_1\cong S^{p-1}\h\otimes (\det)^{-1}$. By Schur's lemma, this homomorphism has to be zero, and thus $\h \otimes (\det)^{-2}$ consists of singular vectors. 

To see these are all singular vectors in $M^1_2$, it is enough to find one nonsingular vector, because the quotient of $M^1_2$ by $\h \otimes (\det)^{-2}$ is an irreducible representation. Direct computation shows that $D_{y_1}x_1x_2\otimes y_1^{p-2}$ has $c_1$ coefficient equal to $x_2\otimes y_1^{p-2}$, which is nonzero in $M^1$.
\end{proof}

The next module to consider is $M^2$, defined as the quotient of $M^1$ by the singular vectors from the previous lemma. The irreducible module $L_{0,c}(S^{p-2}\h)$ is a quotient of $M^2$, and they agree in degrees $0,1,2$. We proceed looking for singular vectors in $M^2_3$.

\begin{lemma}\label{0,p-2,3}
All vectors in $M^2_3\cong S^{p - 3}\h \otimes (\mathrm{det})^{-1}$ are singular. 
\end{lemma}
\begin{proof}

One can explicitly compute that $D_{y_1}(x_1^2x_2\otimes y_1^{p-2}) = 0$ in $M^2_2$ using the parametrization of conjugacy classes from Lemma \ref{reflequiv}.  We claim this implies the lemma.

First, $M^1_3$ is the quotient of $M_{0,c}(S^{p-2}\h)_3 \cong S^3\h^*\otimes S^{p-2}\h$ by the image of singular vectors from Lemma \ref{0,p-2,3}, isomorphic to $S^2\h^*\otimes S^{p-3}\h$. The short exact sequence describing this inclusion is the one from Lemma \ref{tensorprod} combined with Lemma \ref{duals}, giving
$$0 \rightarrow S^{2}\h^*\otimes S^{p-3}\h \rightarrow S^{3}\h^*\otimes S^{p-2}\h \rightarrow S^{p+1} \h \otimes (\det)^{-3} \rightarrow 0.$$
Under these morphisms, the image of $x_1^2x_2\otimes y_1^{p-2}\in S^{3}\h^*\otimes S^{p-2}\h$ in $S^{p+1} \h \otimes (\det)^3$ is $y_1^{p-1}y_2^2$.
 
Second, $M^2_3$ is the quotient of $M^1_3\cong S^{p+1} \h \otimes (\det)^{-3}$ by the image of singular vectors from Lemma \ref{0,p-2,2}, which is the space isomorphic to $\h^*\otimes \h \otimes (\det)^{-2}$. The short exact sequence realizing this inclusion and quotient is the one from Lemma \ref{reducibles}, giving
$$0 \rightarrow \h \otimes \h \otimes (\det)^{-3} \rightarrow S^{p+1} \h \otimes (\det)^{-3} \rightarrow S^{p - 3}\h \otimes (\mathrm{det})^{-1} \rightarrow 0.$$
The image of $y_1^{p-1}y_2^2 \in S^{p+1} \h \otimes (\det)^{-3}$ under this quotient morphism is $-y_1^{p-3}\in S^{p - 3}\h \otimes (\mathrm{det})^{-1}$, which is a nonzero element of it.  

Hence, the image of $x_1^2x_2\otimes y_1^{p-2}$ in $M^2_3$ is the nonzero vector in an irreducible representation. Proving that this vector is singular  will show that the entire space $M^2_3$ consists of singular vectors. 

Third, applying Dunkl operators is a map $\h\otimes M^2_3\to M^2_2$, so let us decompose $\h\otimes M^2_3\cong \h\otimes S^{p - 3}\h \otimes (\mathrm{det})^{-1}$. The short exact sequence doing this is the one from Lemma \ref{tensorprod}, 
$$0 \rightarrow  S^{p-4}\h  \rightarrow \h\otimes S^{p - 3}\h \otimes (\mathrm{det})^{-1} \rightarrow S^{p-2}\h \otimes (\mathrm{det})^{-1} \rightarrow 0.$$ As applying Dunkl operator maps this to $M^2_2\cong S^{p-2}\h\otimes   (\det)^{-1}$, by Schur's lemma the submodule  $S^{p-4}\h$ maps to zero, and the map is zero on $\h\otimes S^{p - 3}\h \otimes (\mathrm{det})^{-1}$ if and only if it is zero on the quotient $S^{p-2}\h \otimes (\mathrm{det})^{-1}$. Under this quotient map, the vector $y_1\otimes (-y_1^{p-2})$ maps to $-y_1^{p-2}$, which is a nonzero element of the irreducible representation $S^{p-2}\h\otimes (\det)^{-1}$. Showing that the entire $\h\otimes M^2_3$ maps to zero is equivalent to showing that  this vector maps to zero, which is equivalent to showing that the image of $D_{y_1}(x_1^2x_2\otimes y_1^{p-2})$ in $M^2_2$ is zero.

Finally, the computation of  $D_{y_1}(x_1^2x_2\otimes y_1^{p-2}) = 0$ in $M^2_2$ is given by the following.  
The factor of $-c_\lambda$ in $D_{y_1}(x_1^2x_2\otimes y_1^{p-2})$ is
$$\sum_{b,d} \left( x_1^{2}d (\lambda+bd)^2 + x_1x_2 (1-\lambda-bd)(1+\lambda-bd+2bd(1-\lambda-bd))+\right.$$
$$\left.+x_2^2(1-\lambda-bd)^2b(bd-1)\right)\otimes \frac{1}{\lambda^{p-2}}\sum_{i=0}^{p-2}{p-2\choose i}(1-bd)^id^{p-2-i}y_1^{i}y_2^{p-2-i}.$$ The sum is over $b,d\in \F_p$ if $\lambda\ne 1$ and over $b\in \F_p, d\in \F_p^{\times}$ if $\lambda=1$. After quotienting out by vectors from the previous two lemmas, whose images in degree $2$ are $x_i(x_1\otimes y_1+x_2\otimes y_2)f, f\in S^{p-3}\h$, and $x_1^2\otimes y_2^{p-2}$, $x_2^{2}\otimes y_1^{p-2}$, we can write this as 
$$\sum_{b,d}x_1x_2 \frac{1}{\lambda^{p-2}}\otimes \Big( -d(\lambda+bd)^2 \sum_{i=1}^{p-2}{p-2\choose i}(1-bd)^id^{p-2-i}y_1^{i}y_2^{p-2-i}+$$
$$+(1-\lambda-bd)(1+\lambda-bd+2bd(1-\lambda-bd))\sum_{i=0}^{p-2}{p-2\choose i}(1-bd)^id^{p-2-i}y_1^{i}y_2^{p-2-i}-$$
$$-(1-\lambda-bd)^2b(bd-1) \sum_{i=0}^{p-3}{p-2\choose i}(1-bd)^id^{p-2-i}y_1^{i}y_2^{p-2-i}\Big)$$
Reading off the coefficients of $x_1x_2\otimes y_1^iy_2^{p-2-i}$ for all $0\le i\le p-2$ and using lemma \ref{ntuplesum} multiple times, we see this is indeed $0$.
\end{proof}
This completes the proof of Proposition \ref{0,p-2}.

\subsection{Characters of $L_{t,c}(S^{i}\h)$ for $i=p-2$ and $t=1$}

\begin{prop}\label{1,p-2}
The reduced character of $L_{1,c}(S^{p-2}\h)$ for generic value of $c$ is
$$H(z)=[S^{p-2}\h]+[S^{p-1}\h\otimes (\det)^{-1}]z+[S^{p-2}\h \otimes (\det)^{-1}]z^{2},$$ so its character and Hilbert polynomial are
$$\chi_{L_{1,c}(S^{p-2}\h)}(z)=([S^{p-2}\h]+[S^{p-1}\h\otimes (\det)^{-1}]z^p+[S^{p-2}\h (\det)^{-1}]z^{2p})\cdot \chi_{S^{(p)}\h^*}(z),$$
$$\mathrm{Hilb}_{L_{1,c}(S^{p-2}\h)}(z)=\left( (p-1)+p{z}+(p-1)z^{2p}\right) \frac{(1-z^p)^2}{(1-z)^2}.$$
\end{prop}
\begin{proof}
It is explained in Proposition \ref{generic}, Corollary \ref{newp} and comments between them that the generators of the module $J_{1,c}(\tau)$ for generic $c$ are in degrees divisible by $p$. If $J_{<mp}$ is the submodule of $J$ generated under $S\h^*$ by elements of $J_{1,c}(\tau)$ of degrees $<mp$, then the generators in degree $mp$ have nonzero projections to $M_{1,c}(\tau)/J_{<mp}$, are singular vectors in this quotient, and form a subrepresentation whose composition factors are a subset of composition factors of $(S^m\h^*)^p\otimes \tau/((S^m\h^*)^p\otimes \tau\cap J_{<mp})$. In Lemmas \ref{1,p-2,1,sing}, \ref{1,p-2,2,sing},  and \ref{1,p-2,3,sing} below we explicitly find these generators for $\tau=S^{p-2}\h$, and in Lemmas \ref{1,p-2,1,nonsing} and \ref{1,p-2,2,nonsing} we prove they are the only ones in degrees $p$, $2p$ and $3p$. The quotient of the Verma module $M_{1,c}(S^{p-2}\h)$ by the submodule generated by these elements is finite dimensional, and zero in degree $4p$, from which we conclude that they generate the whole $J_{1,c}(S^{p-2}\h)$ for generic $c$, and that this quotient is irreducible. 

The reduced character is calculated in the way explained after Proposition \ref{generic}: as we know the generators of $J_{1,c}(S^{p-2}\h)$ explicitly, we evaluate them at $c=0$ (in fact, they do not depend on $c$). They are of the form $f_i(x_1^p,x_2^p)\otimes v_i$. The reduced module is then defined to be $S\h^*\otimes S^{p-2}\h/\left< f_i(x_1,x_2)\otimes v_i \right>$. In our case, the generators $f_i(x_1^p,x_2^p)\otimes v_i$ form subrepresentations of type $S^{p-3}\h$ in degree $p$, $\h\otimes (\det)^{-2}$ in degree $2p$, and $S^{p-3}\h \otimes (\det)^{-1}$ in degree $3p$. The quotient of $S\h^*\otimes S^{p-2}\h$ by $\left< f_i(x_1,x_2)\otimes v_i \right>$ is thus equal to the quotient by subrepresentations of type $S^{p-3}\h$ in degree $1$, $\h\otimes (\det)^{-2}$ in degree $2$, and $S^{p-3}\h \otimes (\det)^{-1}$ in degree $3$, which is easily seen to have the character 
$$H(z)=[S^{p-2}\h]+[S^{p-1}\h\otimes (\det)^{-1}]z+[S^{p-2}\h (\det)^{-1}]z^{2}.$$ 
\end{proof}

\begin{lemma}\label{1,p-2,1,sing}
The vectors $$x_1^p\otimes y_1f+x_2^p\otimes y_2f$$  in $S^{p}\h^*\otimes S^{p-2}\h\cong M_{1,c}(S^{p-2}\h)_p$ are singular in $M_{1,c}(S^{p-2}\h)$ for all  $f\in S^{p-3}\h$. They form a $GL_{2}(\F_p)$ subrepresentation of $M_{1,c}(S^{p-2}\h)_p$ isomorphic to $S^{p-3}\h$.
 \end{lemma}

\begin{proof}
The space of these vectors are symmetric with respect to switching indices $1$ and $2$, so it is enough to prove $D_{y_1}$ acts on it by zero. A computation very similar to the one in Lemma \ref{0,p-2,1} gives 
$$D_{y_1}(x_1^p\otimes y_1f+x_2^py_2f)=-\sum_{\lambda}c_{\lambda} \sum_{b,d}(x_1+bx_2)^{p-1}\otimes \frac{1}{\lambda}((1-\lambda-bd)y_1+dy_2)(s.f).$$
The inner sum is over all $b,d\in \F_p$ if $\lambda\ne 1$ and over all $b\in \F_p, d\in \F_p^{\times}$ if $\lambda=1$. However, if $\lambda=1$, then every summand is divisible by $d$ so the term with $d=0$ does not  contribute, and we can also consider it as a sum over all $d\in \F_p$. The degree in $d$ of every term of this polynomial is less or equal to $1+\deg f=p-2<p-1$, so by Lemma 2.26 of \cite{BalChen}, the sum is zero. 
\end{proof}

\begin{lemma} \label{1,p-2,1,nonsing}
For generic $c$, the vectors from Lemma \ref{1,p-2,1,sing} are the only singular vectors in $M_{1,c}(S^{p-2}\h)_p$.
\end{lemma}
\begin{proof}
We will use Corollary \ref{newp}, which in our situation says that the space of singular vectors in $M_{1,c}(S^{p-2}\h)_p$ has the same composition series (as a $GL_{2}(\F_p)$ representation) as some subrepresentation of $(\h^*)^p\otimes S^{p-2}\h$. This space is isomorphic to $\h^*\otimes S^{p-2}\h$ and fits into a short exact sequence from Lemma \ref{tensorprod}
$$0 \rightarrow S^{p-3}\h \rightarrow \h^*\otimes S^{p-2}\h \rightarrow S^{p-1}\h\otimes (\det)^{-1} \rightarrow 0,$$ so the space of singular vectors in degree $p$ can either be isomorphic to $S^{p-3}\h$ or to its extension by $S^{p-1}\h\otimes (\det)^{-1} $.

We are going to show that the quotient of $(\h^*)^p \otimes S^{p-2}\h$ by $S^{p-3}\h$ isomorphic to $S^{p-1}\h\otimes (\det)^{-1} $ does not consist of singular vectors, and that there is no other composition factor of $S^p\h^*\otimes S^{p-2}\h$ isomorphic to $S^{p-1}\h\otimes (\det)^{-1} $. This will  prove the claim. 

First, we claim that there is a vector in $(\h^*)^p\otimes S^{p-2}\h$ which is not singular. Namely, we claim that $D_{y_1}(x_1^p\otimes y_1^{p-2}) \ne 0$. From the calculation in the proof of Lemma \ref{1,p-2,1,sing} we can read off that
the coefficient of $c_1 x_1^{p-1} \otimes y_1^{p-2}$ is:
$$\sum_{b\in \F_p, d\in \F_p^{\times}}bd (1-bd)^{p-2}=1\ne 0.$$

Second, we claim that $S^p\h^*\otimes S^{p-2}\h$ has only one composition factor of type $S^{p-1}\h\otimes (\det)^{-1} $. The quotient of $S^p\h^*\otimes S^{p-2}\h$ by the space $(\h^*)^p\otimes S^{p-2}\h$ which we already considered is isomorphic to $S^{p-2}\h \otimes S^{p-2}\h$, which by repeated use of Lemma \ref{tensorprod} has subquotients of the form $S^{2p-4-2j}\h \otimes (\det)^j$. More precisely, the expression in the Grothendieck group $K_{0}(GL_{2}(\F_p))$  for $[S^{p-2}\h \otimes S^{p-2}\h]$ is 
 $$[S^{2p-4}\h]+[S^{2p-6}\h \otimes (\det)]+\ldots +[S^{2p-4-2j}\h \otimes (\det)^j]+ \ldots +[S^{0}\h \otimes (\det)^{p-2}].$$ Some of these are irreducible (the ones with $2p-4-2j<p$), and the others decompose further using Lemma \ref{reducibles} and \ref{tensorprod} 
 $$[S^{2p-4-2j}\h \otimes (\det)^j]=[S^{p-4-2j}\h\otimes \h \otimes (\det)^j]+[S^{2j+2}\h \otimes (\det)^{j+1+p-4-2j}] =$$
$$=[S^{p-5-2j}\h\otimes \h \otimes (\det)^{j+1}]+[S^{p-3-2j}\h\otimes \h \otimes (\det)^{j}]  +[S^{2j+2}\h \otimes (\det)^{j+1+p-4-2j}].$$
These are all the composition factors, and none of them are equal to $S^{p-1}\h\otimes (\det)^{-1}$.
\end{proof}

Next, we consider the auxiliary module ${\bf M}^1$, defined as the quotient of the Verma module $M_{1,c}(S^{p-2}\h)$ by the submodule generated by singular vectors $(x_1^p\otimes y_1+x_2^p\otimes y_2) f$ from Lemma \ref{1,p-2,1,sing}. This module ${\bf M}^1$ matches $L_{1,c}(S^{p-2}\h)$ in degrees $0,1,\ldots 2p-1$, and we search for singular vectors in ${\bf M}^1_{2p}$.

\begin{lemma}\label{1,p-2,2,sing}
The images of the vectors $$x_1^{2p}\otimes y_2^{p-2},\,\,\,\, x_2^{2p}\otimes y_1^{p-2}$$ in $\bf{M}^1$ are singular, and span a subrepresentation isomorphic to $\h \otimes (\det)^{-2}$.
\end{lemma}

\begin{proof}
It is easy to check  that they span an irreducible subrepresentation isomorphic to $\h \otimes (\det)^{-2}$. It is then a direct computation to show that for one of the vectors, for example $x_1^{2p}\otimes y_2^{p-2}$, satisfies that $D_{y_1}(x_1^{2p}\otimes y_2^{p-2})$ and $D_{y_2}(x_1^{2p}\otimes y_2^{p-2})$ are in the submodule spanned by the singular vectors from Lemma \ref{1,p-2,1,sing}. Thus, $x_1^{2p}\otimes y_2^{p-2}$ is singular in the quotient $\bf{M}^1$.

The computation is as follows.  It is enough to show that one of them is singular, for example $x_1^{2p}\otimes y_2^{p-2}$. An explicit computation yields a nonzero coefficient of $-c_{\lambda}$ in $D_{y_1}(x_1^{2p}\otimes y_2^{p-2})$, and denoting $\alpha_b=x_1+bx_2$, we get it is equal to:
$$\sum_{b,d}\frac{1}{\alpha_b}\left( x_1^{2p}-(x_1-(1-\lambda-bd)\alpha_b)^{2p} \right)\otimes \frac{(b(1-\lambda-bd)y_1+(\lambda+bd)y_2)^{p-2}}{\lambda^{p-2}}$$
$$=\sum_{b,d}(1-\lambda-bd)\left( (1+\lambda+bd)x_1^{p}-b(1-\lambda-bd)x_2^p \right)\left( \sum_{i=0}^{p-1}(-1)^ib^ix_1^{p-1-i}x_2^i\right) \otimes$$
$$\otimes \frac{1}{\lambda^{p-2}}\sum_{j=0}^{p-2}{p-2\choose j}b^j(1-\lambda-bd)^j(\lambda+bd)^{p-2-j}y_1^jy_2^{p-2-j}.$$
Summing over all $d$ and using lemma \ref{ntuplesum}, this is equal to
$$ \sum_{b}\left(2\lambda bx_1^p-2b^2(1-\lambda)x_2^p \right)\left(\sum_{i=0}^{p-1}(-1)^ib^ix_1^{p-1-i}x_2^i \right)\otimes \frac{1}{\lambda^{p-2}}\sum_{j=0}^{p-2}{p-2\choose j}(-1)^j b^{j+p-2}y_1^jy_2^{p-2-j}+$$
$$+\sum_{b}\left( b^2x_1^p+b^3x_2^p\right)\left(\sum_{i=0}^{p-1}(-1)^ib^ix_1^{p-1-i}x_2^i \right)\otimes \frac{1}{\lambda^{p-2}}\Bigg( -2\lambda b^{p-3}y_2^{p-2}-2(1-\lambda)b^{2p-5}y_1^{p-2}+$$
$$+\sum_{j=1}^{p-3}{p-2 \choose j} (-1)^jb^{p-3-j}(-2\lambda-j)y_1^jy_2^{p-2-j}\Bigg) .$$
Summing over $b\in \F_p$ using lemma \ref{ntuplesum} and reorganizing the terms, we get
$$\frac{-1}{\lambda^{p-2}}\Bigg( 2(1-\lambda)\left( x_1^{2p-2}x_2-x_1^{p-1}x_2^p\right)\otimes y_1^{p-2}+2\lambda\left( x_1x_2^{2p-2}-x_1^{p}x_2^{p-1}\right)\otimes y_2^{p-2}+$$
$$+2\lambda x_1^{2p-2}x_2 \otimes y_1^{p-2}+2\lambda x_1^{p}x_2^{p-1}\otimes y_2^{p-2}+2(1-\lambda) x_1^{p-1}x_2^p \otimes y_1^{p-2}+2(1-\lambda) x_1x_2^{2p-2} \otimes y_2^{p-2}+$$
$$+\sum_{i=1}^{p-3}{p-2 \choose i}\left( (2+i)x_1^{2p-2-i}x_2^{i+1}-ix_1^{p-1-i}x_2^{p+i}\right)\otimes y_1^{p-2-i}y_2^{i} \Bigg).$$
This is nonzero in $M_{1,c}(S^{p-2}\h)$, and
so the vectors from lemma are not singular there. However, in the quotient ${\bf M}^1$, the expression is zero.
This can be shown by replacing, in the above long expression, anything of the form $a\cdot x_2^p\otimes y_2 \cdot f$ (meaning any term whose degree of $x_2$ is at least $p$ and whose degree of $y_2$ is at least $1$) by the equivalent expression $-a\cdot x_1^p \otimes y_1 \cdot f$. The expression then simplifies to $0$, showing that in the quotient ${\bf M}^1$,  $D_{y_1}(x_1^{2p}\otimes y_2^{p-2})=0$.
Similarly, the coefficient in $D_{y_2}(x_1^{2p}\otimes y_2^{p-2})$ of $-c_{\lambda}$ is equal to 
$$\sum_{b,d}\frac{b}{\alpha_b}\left( x_1^{2p}-(x_1-(1-\lambda-bd)\alpha_b)^{2p} \right)\otimes \frac{(b(1-\lambda-bd)y_1+(\lambda+bd)y_2)^{p-2}}{\lambda^{p-2}}+$$
$$+\sum_{a}\frac{1}{x_2}\left(x_1^{2p}-(x_1-ax_2)^{2p} \right)\otimes \frac{(ay_1+y_2)^{p-2}}{\lambda^{p-2}}.$$
The proof that the first part of the sum is $0$ in ${\bf M}^1$ is very similar to the previous computation, as it differs from the expression calculated there just by one power of $b$. The second part is equal to $$\frac{-1}{\lambda^{p-2}}\left( 2x_1^px_2^{p-1}\otimes y_1^{p-2}+2x_2^{2p-1}\otimes y_1^{p-3}y_2 \right),$$ which is also zero in ${\bf M}^1$.
A similar argument works for $D_{y_2}(x_1^{2p} \otimes y_2^{p-2})$.

\end{proof}

\begin{lemma} \label{1,p-2,2,nonsing}
For generic $c$, the vectors from Lemma \ref{1,p-2,2,sing} are the only singular vectors in ${\bf M}^1_{2p}$.
\end{lemma}
\begin{proof}
The space of $p$-th powers $(S^{2}\h^*)^p\otimes S^{p-2}\h$ in $M_{1,c}(S^{p-2}\h)_{2p}\cong  S^{2p}\h^*\otimes S^{p-2}\h$ is isomorphic to $S^{2}\h^*\otimes S^{p-2}\h$. Using Lemma \ref{tensorprod}, its image in the quotient  ${\bf M}^1_{2p}$ is isomorphic to $(S^{2}\h \otimes S^{p-2}\h \otimes (\det)^{-2})/ (\h\otimes S^{p-3}\h\otimes (\det)^{-1}) \cong S^{p}\h\otimes (\det)^{-2}$. By Lemma \ref{reducibles}, this decomposes as
$$0 \rightarrow  \h \otimes (\det)^{-2} \rightarrow S^{p} \h \otimes (\det)^{-2} \rightarrow S^{p - 2}\h \otimes (\det)^{-1} \rightarrow 0.$$
We already showed in that the subspace $  \h \otimes (\det)^{-2} $ consists of singular vectors. By Lemma 
\ref{newp}, to see these are all singular vectors in this degree, we need to show that there is no extension of $ \h \otimes (\det)^{-2} $ by $S^{p - 2}\h \otimes (\det)^{-1}$ consisting of singular vectors.  

First, direct computation shows that that $x_1^px_2^p\otimes y_1^{p-1}$ is a $p$-th power that is not annihilated by $D_{y_1}$, so the entire space $S^{p} \h \otimes (\det)^{-2}$ consists of singular vectors. 

Secondly, let us show that there is no other composition factor of ${\bf M}^1_{2p}$ isomorphic to $S^{p - 2}\h \otimes (\det)^{-1}$ and extending $ \h \otimes (\det)^{-2}$. If there was, it would be in the socle of the quotient of ${\bf M}^1_{2p}$ by $p$-th powers, which is, by Lemmas \ref{tensorprod} and \ref{reducibles}, and \cite{glover1}, Theorem (5.3), isomorphic to 
$$ S^{p-2}\h\otimes S^{p-1}\h \otimes (\det)^{-1}\cong S^{p(p-1)-1}\h \otimes (\det)^{-1}.$$
By Theorem (5.9) in \cite{glover1} its socle is
$$\bigoplus_{m=0}^{(p-3)/2} S^{2m+1}\h \otimes (\det)^{p-3-m}.$$ As these are all irreducible modules and none of them is isomorphic to $S^{p - 2}\h \otimes (\det)^{-1}$, we conclude there are no other singular vectors in  ${\bf M}^1_{2p}$.

\end{proof}

The second auxiliary module ${\bf M}^2$ to consider is the quotient of $\bf{M}^1$ by the rational Cherednik algebra submodule generated by the two dimensional space of singular vectors from the previous lemma. Because of Corollary \ref{newp} and the previous lemmas in this section (Lemma \ref{1,p-2,1,nonsing} and \ref{1,p-2,2,nonsing}), ${\bf M}^2$ is equal to $L_{1,c}(S^{p-2}\h) $ in degrees up to $3p-1$. In degree $3p$, ${\bf M}^2$ contains some new singular vectors, given in the following lemma.

\begin{lemma}\label{1,p-2,3,sing}
The subspace $(S^3\h^*)^p\otimes S^{p-2}\h \subseteq M_{1,c}(S^{p-2}\h)_{3p} $ is in $\Ker B$.
\end{lemma}
\begin{proof}
The space of $p$-th powers in $M_{1,c}(S^{p-2}\h)_{3p}$ is $(S^3\h^*)^p\otimes S^{p-2}\h$. The quotient of this space by the subspace $(S^2\h^*)^p\otimes S^{p-3}\h$ generated by singular vectors from Lemma \ref{1,p-2,1,sing} is by Lemma \ref{tensorprod} isomorphic to $S^{p+1}\h \otimes (\det)^{-3}$. The quotient of this space by the subspace $(\h^*)^p\otimes \h\otimes (\det)^{-2}$ generated by singular vectors from Lemma \ref{1,p-2,2,sing} is by Lemma \ref{reducibles} isomorphic to $S^{p-3}\h \otimes (\det)^{-1}$, and this is the space of $p$-th powers in the ${\bf M}^2_{3p}$. 

We are going to show that this space consists of singular vectors by showing that the restriction to $\h \otimes S^{p-3}\h \otimes (\det)^{-1}$ of the map $\h\otimes {\bf M}^2_{3p}\to {\bf M}^2_{3p-1}$ given by $y\otimes m\mapsto D_y(m)$, which is a homomorphism of group representations, has to be zero. 

The source space of this map is $\h \otimes S^{p-3}\h \otimes (\det)^{-1}$, which fits into the short exact sequence 
$$0 \rightarrow S^{p-4}\h \rightarrow \h \otimes S^{p-3}\h \otimes (\det)^{-1} \rightarrow S^{p-2}\h \otimes (\det)^{-1} \rightarrow  0.$$ The image of the map is a subrepresentation of the target space  ${\bf M}^2_{3p-1}$, so if we show that the socle of ${\bf M}^2_{3p-1}$ does not  have $S^{p-4}\h$ nor $S^{p-2}\h \otimes (\det)^{-1} $ as direct summands, it will follow that the map is zero.

By applying Lemma \ref{reducibles} twice and Lemma \ref{tensorprod} once, we see that the quotient of  ${\bf M}^2_{3p-1}$ by the image $S^{2p-1}\h^*\otimes S^{p-3}\h$ of singular vectors from Lemma \ref{1,p-2,1,sing} is isomorphic to $S^{p-1}\h\otimes S^{p}\h \otimes (\det)^{-2}$. The quotient of that by the image $S^{p-1}\h^*\otimes \h \otimes (\det)^{-2}$ is by Lemma \ref{reducibles} isomorphic to $S^{p-1}\h \otimes S^{p-2}\h \otimes (\det)^{-1}$.  Using \cite{glover1}, Theorem (5.3) again, this is isomorphic to $S^{p(p-1)-1}\h \otimes (\det)^{-1},$ whose socle is by Theorem (5.9) in \cite{glover1} again equal to
$$\bigoplus_{m=0}^{(p-3)/2} S^{2m+1}\h \otimes (\det)^{p-3-m}.$$ None of these summands is of the type  $S^{p-4}\h$ nor $S^{p-2}\h \otimes (\det)^{-1} $, so the required map is zero. This proves the lemma. 

\end{proof}

\section{Characters of $L_{t,c}(S^{i}\h)$ for $i=p-1$}
\subsection{Characters of $L_{t,c}(S^{i}\h)$ for $i=p-1$ and $t=0$}

In this section we will calculate the character of $L_{0,c}(S^{p-1}\h)$ for generic $c$. We first find a certain space $\mathrm{span}_{\Bbbk}\{v_0,\ldots ,v_{p-1} \}$ of singular vectors in $M_{0,c}(S^{p-1}\h)_{p-1}$. We define an auxiliary module $M$ as a quotient of the baby Verma module $N_{0,c}(S^{p-1}\h)$ by the $H_{0,c}(GL_{2},\h)$-submodule generated by these singular vectors. Alternatively, $M$ is the quotient of $M_{0,c}(S^{p-1}\h)$ by the submodule generated by $\mathrm{span}_{\Bbbk}\{v_0,\ldots ,v_{p-1} \}$, $Q_0\otimes S^{p-1}\h$ and $Q_0\otimes S^{p-1}\h$  (for definitions of the invariants $Q_0,Q_1\in (Sh^*)^{G}$, see section \ref{babychar}). We calculate the character of $M$, and finally we show that $M$ is irreducible and isomorphic to $L_{0,c}(S^{p-1}\h)$.

We will extensively use Corollary \ref{1block}, which states that all the composition factors of $M_{t,c}(S^{p-1}\h)$ are of the form $L_{t,c}(S^{p-1}\h)$. Because of that, all the singular vectors in $M_{t,c}(S^{p-1}\h)$ or any of its subrepresentations or quotients have all of their composition factors isomorphic to $S^{p-1}\h$.

\begin{lemma}\label{oneissing}
The vector $$x_1^{p-1}\otimes y_2^{p-1}-x_2^{p-1}\otimes y_1^{p-1}$$ in $S^{p-1}\h^*\otimes S^{p-1}\h \cong M_{0,c}(S^{p-1}\h)_{p-1}$ is singular. 
\end{lemma}
\begin{proof}
The proof is pure computation, using the parametrization of conjugacy classes from Lemma \ref{reflequiv} and Lemma \ref{ntuplesum} extensively.

$$\lambda\ne 1:\,\,\,\, C_{\lambda}=\{\left[\begin{array}{c} 1 \\ b \end{array} \right] \otimes \left[\begin{array}{c} 1-\lambda -bd \\ d \end{array} \right]  | b,d\in \F_p  \} \cup \{ \left[\begin{array}{c} 0 \\ 1 \end{array} \right] \otimes \left[\begin{array}{c} a \\ 1-\lambda \end{array} \right]  | a \in \F_p \} $$
$$\lambda= 1:\,\,\,\, C_{1}=\{\left[\begin{array}{c} 1 \\ b \end{array} \right] \otimes \left[\begin{array}{c} -bd \\ d \end{array} \right]  | b,d\in \F_p, d\ne 0  \} \cup \{ \left[\begin{array}{c} 0 \\ 1 \end{array} \right] \otimes \left[\begin{array}{c} a \\ 0 \end{array} \right]  | a \in \F_p, a\ne 0 \}.$$

First, this vector is antisymmetric with respect to indices $1,2$, so it is enough to show that $D_{y_1}$ acts on it by $0$. For any $\lambda$ the coefficient of $-c_{\lambda}$ in $D_{y_1}(x_1^{p-1}\otimes y_2^{p-1}-x_2^{p-1}\otimes y_1^{p-1})$ is 
\begin{equation} \tag{$\star$} \sum_{s\in C_{\lambda}}(y_1,\alpha_s)\left( \frac{x_1^{p-1}-s.x_1^{p-1}}{\alpha_s}\otimes (s.y_2)^{p-1}-\frac{x_2^{p-1}-s.x_2^{p-1}}{\alpha_s}\otimes (s.y_1)^{p-1} \right). \end{equation}

Let us rewrite this using the parametrization of $C_{\lambda}$ from Lemma \ref{reflequiv}. We use notation $\alpha_b=\left[ \begin{array}{c} 1 \\ b \end{array} \right] \in \h^*$. The sum is over all $b,d\in \F_{p}$ if $\lambda\ne 1$ and over all $b,d\in \F_p$, $d\ne 0$, if $\lambda=1$. The above expression is equal to:

\begin{eqnarray*}
(\star)&=&\sum_{b,d}\frac{1}{\alpha_b}\left(x_1^{p-1}-(x_1-(1-\lambda-bd)\alpha_b)^{p-1}\right)\otimes \frac{1}{\lambda^{p-1}}\left( b(1-\lambda-bd)y_1+(\lambda+bd)y_2 \right)^{p-1} +\\
&&+\frac{1}{\alpha_b}\left((x_2-d\alpha_b)^{p-1} -x_2^{p-1}\right)\otimes \frac{1}{\lambda^{p-1}}\left( (1-bd)y_1+dy_2 \right)^{p-1}=\\
&=& \frac{1}{\lambda^{p-1}}\sum_{b,d}\sum_{i=1}^{p-1}{p-1\choose i} (-1)^{i+1}x_1^{p-1-i}(1-\lambda-bd)^i\alpha_b^{i-1} \otimes  \left( b(1-\lambda-bd)y_1+(\lambda+bd)y_2 \right)^{p-1}+ \\
&& + {p-1\choose i} (-1)^{i}x_2^{p-1-i}d^i\alpha_b^{i-1}\otimes   \left( (1-bd)y_1+dy_2 \right)^{p-1}=  \\
&=& \frac{1}{\lambda^{p-1}}\sum_{b,d}\sum_{i=1}^{p-1}\sum_{j=0}^{i-1}\sum_{k=0}^{p-1} {p-1\choose i} {i-1\choose j} {p-1\choose k}\cdot \\
&& \cdot \left( (-1)^{i+1}(1-\lambda-bd)^ib^jx_{1}^{p-j-2} x_2^j\otimes b^{p-1-k}(1-\lambda-bd)^{p-1-k}(\lambda+bd)^ky_1^{p-1-k}y_2^k + \right. \\
&& \left. + (-1)^i d^ib^{i-1-j} x_1^j x_2^{p-j-2}\otimes (1-bd)^{p-1-k}d^ky_1^{p-1-k}y_2^k \right) =\\
&=& \frac{1}{\lambda^{p-1}}\sum_{k=0}^{p-1}\sum_{j=0}^{p-2}{p-1\choose k}x_{1}^{p-j-2} x_2^j\otimes y_1^{p-1-k}y_2^k \cdot \\
&&\cdot \sum_{b,d} \left( \sum_{i=j+1}^{p-1}{i-1 \choose j}{p-1 \choose i}(-1)^{i+1}(1-\lambda-bd)^{p-1-k+i}(\lambda+bd)^{k}b^{p-1-k+j} \right. +\\
&& + \left. \sum_{i=p-1-j}^{p-1}{p-1 \choose i}{i-1 \choose p-2-j}(-1)^id^{i+k}b^{i+1+j-p}(1-bd)^{p-1-k} \right) .
\end{eqnarray*}

Reading off the coefficient of $x_1^{p-j-2}x_2^{j}\otimes y_1^{p-1-k}y_2^k$ and using that $\frac{1}{\lambda^{p-1}}{p-1 \choose k}$ is never zero, the claim that $(\star)=0$ is equivalent to showing that for every $0\le k \le p-1$, $0\le j\le p-2$, the expression $(\star \star)$ is zero, where $(\star \star)$ is 
\begin{eqnarray*}
&& \sum_{b,d} \left( \sum_{i=j+1}^{p-1}{i-1 \choose j}{p-1 \choose i}(-1)^{i+1}(1-\lambda-bd)^{p-1-k+i}(\lambda+bd)^{k}b^{p-1-k+j} \right. +\\
&& + \left. \sum_{i=p-1-j}^{p-1}{p-1 \choose i}{i-1 \choose p-2-j}(-1)^id^{i+k}b^{i+1+j-p}(1-bd)^{p-1-k} \right) \\
&=& \sum_{b,d} \left( \sum_{i=j+1}^{p-1}{i-1 \choose j}{p-1 \choose i}(-1)^{i+1}(1-\lambda-bd)^{p-1-k+i}(\lambda+bd)^{k}b^{p-1-k+j} \right. +\\
&& + \left. \sum_{i=0}^{j}{p-1 \choose p-1-i}{p-2-i \choose p-2-j}(-1)^id^{p-1+k-i}b^{j-i}(1-bd)^{p-1-k} \right) \\
&=& \sum_{b,d} \left( \sum_{i=j+1}^{p-1}\sum_{m=0}^{p-1-k+i}\sum_{n=0}^{k} {i-1 \choose j}{p-1 \choose i} {p-1-k+i \choose m}{k \choose n}(-1)^{m+i+1} \cdot \right. \\
&& \left. \cdot (1-\lambda)^{p-1-k+i-m}\lambda^{k-n} b^{m+n+p-1-k+j}d^{m+n} + \right. \\
&&  \left. \sum_{i=0}^{j}\sum_{l=0}^{p-1-k} {p-1 \choose p-1-i}{p-2-i \choose p-2-j} 
 {p-1-k \choose l} (-1)^{i+l} b^{j-i+l}d^{p-1+k-i+l}\right) 
\end{eqnarray*}

Now we will use Lemma \ref{ntuplesum}, which states that $\sum_{b\in \F_p}b^N=0$ and $\sum_{d\in \F_p}b^N=0$, unless $N\equiv 0 \pmod{p-1}$.

First assume $\lambda\ne 1$. The first part of the sum includes $\sum_b b^{m+n+p-1-k+j}$ and $\sum_{d} d^{m+n}$, so it is zero unless $$m+n \equiv 0 \pmod{p-1}$$ $$m+n+p-1-k+j \equiv 0 \pmod{p-1},$$ which implies $$j \equiv k \pmod{p-1}.$$ The second part of the sum includes $\sum_{b,d} b^{j-i+l}d^{p-1+k-i+l}$, so it is zero unless $$j-i+l \equiv 0 \pmod{p-1}$$ $$p-1+k-i+l \equiv 0 \pmod{p-1},$$ which again implies $$j \equiv k \pmod{p-1}.$$

As $0\le k \le p-1$, $0\le j\le p-2$, the possibilities for $j \equiv k \pmod{p-1}$ are $j=0,k=p-1$ or $j=k$. Let us calculate $(\star \star)$ in those two cases separately. 

If $j=0,k=p-1$, then 
\begin{eqnarray*}
(\star \star) &=&  \sum_{b,d} \sum_{i=1}^{p-1}\sum_{m=0}^{i}\sum_{n=0}^{p-1}{p-1 \choose i} {i \choose m}{p-1 \choose n}(-1)^{m+i+1}  (1-\lambda)^{i-m}\lambda^{p-1-n} b^{m+n}d^{m+n} +  \\
&& + \sum_{b,d} d^{2(p-1)} =\\
&=&  \sum_{b,d} \sum_{i=1}^{p-1}\sum_{m=0}^{i}\sum_{n=0}^{p-1}{p-1 \choose i} {i \choose m}{p-1 \choose n}(-1)^{m+i+1}  (1-\lambda)^{i-m}\lambda^{p-1-n} b^{m+n}d^{m+n} =\\
&=& \sum_{b,d} (\lambda-bd)^{2(p-1)}-(\lambda-bd)^{p-1}=0.
\end{eqnarray*}

If $j=k$, then, using that $a^{p}=a$,
\begin{eqnarray*}
(\star \star) &=& \sum_{b,d} \left( \sum_{i=j+1}^{p-1}{i-1 \choose j}{p-1\choose i}(-1)^{i+1}(1-\lambda-bd)^{i-j}b^{p-1}(\lambda+bd)^j \right. + \\
&& +\left. \sum_{i=0}^{j}{p-1\choose i}{p-2-i\choose p-2-j}(-1)^i d^{p-1+j-i}b^{j-i}(1-bd)^{p-1-j} \right) \\
&=& \sum_{b,d} \sum_{i=j+1}^{p-1}\sum_{m=0}^{i-j}\sum_{n=0}^{j}{i-1 \choose j}{p-1\choose i}{i-j \choose m}{j\choose n}(1-\lambda)^{i-j-m}\lambda^{j-n}(-1)^{m+i+1} b^{p-1+m+n}d^{m+n} +  \\
&& + \sum_{b,d}\sum_{i=0}^{j}\sum_{l=0}^{p-1-j}{p-1\choose i}{p-2-i\choose p-2-j}{p-1-j\choose l}(-1)^{l+i}b^{j-i+l}d^{p-1+j-i+l} \\
\end{eqnarray*}
Again using that $\sum_{b\in \F_p}b^N=0$ unless $N\equiv 0 \pmod{p-1}$, $N\ne 0$, this is equal to:
\begin{eqnarray*}
(\star \star) &=& \sum_{b,d} {p-2 \choose j} (-1)^{j+1} b^{2(p-1)}d^{p-1} + \sum_{b,d}{p-2\choose j}(-1)^{j} b^{p-1}d^{2(p-1)} \\
&=& 0.
\end{eqnarray*}

Let us now do a very similar computation for $\lambda=1$. Now $\sum_{b,d}$ is over $b,d\in \F_{p}$, $d\ne 0$.
\begin{eqnarray*}
(\star \star)&=& \sum_{b,d} \left( \sum_{i=j+1}^{p-1}\sum_{n=0}^{k} {i-1 \choose j}{p-1 \choose i} {k \choose n}(-1)^{p-k} b^{2(p-1-k)+i+n+j}d^{p-1-k+i+n} + \right. \\
&&  \left. +\sum_{i=0}^{j}\sum_{l=0}^{p-1-k} {p-1 \choose p-1-i}{p-2-i \choose p-2-j} 
 {p-1-k \choose l} (-1)^{i+l} b^{j-i+l}d^{p-1+k-i+l}\right) 
\end{eqnarray*}
Again, this is zero unless $j\equiv k \pmod{p-1}$. If $j=0,k=p-1$, it is equal to
\begin{eqnarray*}
(\star \star)&=& \sum_{b,d}  \sum_{i=1}^{p-1}\sum_{n=0}^{p-1} {p-1 \choose i} {p-1 \choose n}(-1) b^{i+n}d^{i+n} +  \sum_{b,d} b^{0}d^{2(p-1)} \\
&=& -\sum_{b,d} \left( \sum_{i=1}^{p-1}{p-1 \choose i} {p-1 \choose p-1-i} b^{p-1}d^{p-1} +  b^{2(p-1)}d^{2(p-1)}\right) \\
&=& - \sum_{i=0}^{p-1}{p-1 \choose i}^2 \sum_{b,d} b^{p-1}d^{p-1} \\
&=& - \sum_{i=0}^{p-1}{p-1 \choose i}^2 = 0. \\
\end{eqnarray*}

If $j=k$, 
\begin{eqnarray*}
(\star \star)&=& \sum_{b,d} \left( \sum_{i=j+1}^{p-1}\sum_{n=0}^{j} {i-1 \choose j}{p-1 \choose i} {j \choose n}(-1)^{p-j} b^{2(p-1)-j+i+n}d^{p-1-j+i+n} + \right. \\
&&  \left. +\sum_{i=0}^{j}\sum_{l=0}^{p-1-j} {p-1 \choose p-1-i}{p-2-i \choose p-2-j} 
 {p-1-j \choose l} (-1)^{i+l} b^{j-i+l}d^{p-1+j-i+l}\right) \\
 &=& \sum_{b,d}  {p-2 \choose j} (-1)^{p-j} b^{3(p-1)}d^{2(p-1)} + \\
 && + \sum_{b,d} {p-2 \choose p-2-j}  (-1)^{p-1-j} b^{p-1}d^{2(p-1)}=\\
 &=& {p-2 \choose j} (-1)^{p-j} + {p-2 \choose p-2-j}  (-1)^{p-1-j} = 0.\\
\end{eqnarray*}

So, $(\star \star) =0$ and the vector $x_1^{p-1}\otimes y_2^{p-1}-x_2^{p-1}\otimes y_1^{p-1}$ is singular. 

\end{proof}

\begin{lemma}\label{hbegins}
There are no singular vectors in $M_{0,c}(S^{p-1}\h)_{i}$ for $i<p-1$, and the space of singular vectors in $M_{0,c}(S^{p-1}\h)_{p-1}$ is isomorphic to $S^{p-1}\h$ as a $GL_{2}(\F_p)$-representation. 
\end{lemma}
\begin{proof}

From the previous lemma it follows that the space of singular vectors in $M_{0,c}(S^{p-1}\h)_{p-1}$ is nonzero, and from Lemma \ref{blocks} that all the composition factors of it are isomorphic to $S^{p-1}\h$. We will now show that for $0\le i\le p-1$, $M_{0,c}(S^{p-1}\h)_{i}\cong S^{i}\h^*\otimes S^{p-1}\h$ has no composition factors isomorphic to $S^{p-1}\h$ unless $i=0$ or $i=p-1$, in which case it has one. The claim follows from this. 

Using Proposition \ref{tensorprod},  in the Grothendieck group $K_0(GL_{2}(\F_p))$ we have, for $0\le i \le p-1$: 
\begin{eqnarray*}
[S^{i}\h^*\otimes S^{p-1}\h] &=& [S^{i}\h \otimes S^{p-1}\h \otimes (\det)^{-i}] \\
& = & [S^{p+i-1}\h \otimes (\det)^{-i}] + [S^{p+i-3}\h \otimes (\det)^{-i+1}]+  \ldots  +[S^{p-1-i}\h] .
\end{eqnarray*}
The representation $S^{p-1}\h$ only appears on this list of representations $S^{p+i-1-2j} \h\otimes (\det)^{-i+j}$ for $0\le j\le i$, when $i=0$. Some of the representations on the list are reducible, namely the ones with $i-1-2j\ge 0$. Decomposing them by Proposition \ref{reducibles}, 
$$ [S^{p+i-1-2j} \h \otimes (\det )^{-i+j}] = [ S^{i-1-2j} \h \otimes \h \otimes (\det)^{-i+j}] + [S^{p-1-i+2j}\h \otimes (\det)^{-j}] =$$
$$= [ S^{i-2-2j} \h  \otimes (\det)^{-i+j+1}]+ [S^{i-2j} \h  \otimes (\det)^{-i+j}]+ [S^{p-1-i+2j}\h \otimes (\det)^{-j}]$$
Here, we follow the convention that $S^{k}\h=0$ if $k<0$. In this decomposition all representations are irreducible. Using that $i-1-2j\ge 0$, $p-1\ge i\ge 0$ and $i\ge j\ge 0$, we see that $S^{p-1}\h$ appears on this list only once, namely when $j=0$, $i=p-1$.
\end{proof}

This space is generated by the singular vector from Lemma \ref{oneissing}. Its explicit basis, which we will need in computations below, is given by $v_0,\ldots, v_{p-1}\in S^{p-1}\h^*\otimes S^{p-1}\h$:
$$v_{k}=\sum_{i=0}^{k}(-1)^i x_1^{k-i}x_2^{p-1-k+i}\otimes y_1^{p-1-i}y_2^i+\sum_{i=k}^{p-1}(-1)^i x_1^{p-1+k-i}x_2^{i-k}\otimes y_1^{p-1-i}y_2^i.$$

Remember from section \ref{babychar} that the algebra of invariants $(S\h^*)^{GL_2(\F_p)}$ is a polynomial algebra generated by  $Q_0$ and $Q_1$ of degrees $p^2-p$ and $p^2-1$, constructed explicitly as:
$$Q_0=\left| \begin{array}{cc} x_1^p & x_2^p \\ x_1 & x_2 \end{array} \right|^{p-1}=\left( x_1^px_2-x_1x_2^p)\right)^{p-1}=\left( x_1x_2(x_1^{p-1}-x_2^{p-1})\right)^{p-1}$$
$$Q_1=\frac{\left| \begin{array}{cc} x_1^{p^2} & x_2^{p^2} \\ x_1 & x_2 \end{array} \right| }{\left| \begin{array}{cc} x_1^p & x_2^p \\ x_1 & x_2 \end{array} \right|}=\frac{x_1^{p^2}x_2-x_1x_2^{p^2}}{x_1^px_2-x_1x_2^p}=\frac{x_1^{p^2-1}-x_2^{p^2-1}}{x_1^{p-1}-x_2^{p-1}}=\sum_{i=0}^{p}x_1^{(p-1)i}x_2^{(p-1)(p-i)}.$$
Alternatively, the determinant $\left| \begin{array}{cc} x_1^p & x_2^p \\ x_1 & x_2 \end{array} \right|$ can be described as a product of all the linear polynomials of the form $x_1+ax_2$ and $x_2$. Similarly, $Q_0$ is the product of all the nonzero linear polynomials $ax_1+bx_2$, and $Q_1$ is the product of all the irreducible monic quadratic polynomials $x_1^2+ax_1x_2+bx_2^2$.

Also remember that at $t=0$, the space $(S\h^*)^{GL_2(\F_p)}_{+}\otimes \tau \subseteq M_{0,c}(\tau)$ is always a subspace of $J_{0,c}(\tau)$, and that the spaces $Q_1\otimes \tau$ and $Q_0\otimes \tau$ consist of singular vectors. 


Now define $M$ to be the quotient of $M_{0,c}(S^{p-1}\h)$ by the submodule generated by the singular vectors from the previous two lemmas $\mathrm{span}_{\Bbbk}\{v_0,\ldots ,v_{p-1} \}$, and by the invariants $Q_0\otimes S^{p-1}\h$ and $Q_0\otimes S^{p-1}\h$. To calculate the character of $M$, we investigate the independence of these generators.   

\begin{prop}
The submodule generated by $ Q_{1}\otimes S^{p-1}\h$ is contained in the submodule of generated by $v_0,\ldots v_{p-1}$.

The intersection of the submodule generated by $Q_{0}\otimes S^{p-1}\h$ and the submodule generated by  $v_0,\ldots v_{p-1}$
is generated by $Q_{0}v_0,\ldots Q_0v_{p-1}$ in degree $(p^2-1)(p-1)$.
\end{prop}
\begin{proof}

Let $l=0$ or $1$, and let us study the intersection of the submodule generated by $Q_{l}\otimes S^{p-1}\h$ and the submodule $V$ generated by  $v_0,\ldots v_{p-1}$. This intersection is a graded submodule of $M_{0,c}(S^{p-1}\h)$, containing all elements of the form
\begin{equation*}h_0v_0+h_1v_1+\ldots h_{p-1}v_{p-1}=Q_{l} f,\tag{*}\label{star}\end{equation*}
where $h_{l}(x_1,x_2)\in S^{n}\h^*$ for some $n\ge 0$, and $f\in S^{n+p-1-\deg(Q_l)}\h^*\otimes S^{p-1}\h$. 

One can think of this as a linear equation in $S\h^*\otimes S^{p-1}\h$ with unknowns $h_i$ and $f$. Reading off the coefficients with  $y_1^{p-1-i}y_2^i \in S^{p-1}\h$, we can alternatively think of it as a system of $p$ linear equations in $S\h^*$, with unknowns $h_i$ and $f_i\in S\h^*$ and $f=\sum_i  f_i\otimes y_1^{p-1-i}y_2^i$. The left hand side can then be written as
$$\left[ \begin{array}{cccccc} 
x_1^{p-1}+x_2^{p-1} & x_1x_2^{p-2} & \ldots & x_1^{k}x_2^{p-1-k} & \ldots &  x_1^{p-1}  \\
-x_1^{p-2}x_2 & -(x_1^{p-1}+x_2^{p-1}) & \ldots & -x_1^{k-1}x_2^{p-k} & \ldots &  -x_1^{p-2}x_2\\
x_1^{p-3}x_2^2 & x_1^{p-2}x_2  & \ldots & x_1^{k-2}x_2^{p-k+1} & \ldots & x_1^{p-3}x_2^2\\  
\vdots & \vdots & & & & \vdots \\
(-1)^k x_1^{p-1-k}x_2^k & (-1)^kx_1^{p-k}x_2^{k-1} & \ldots & (-1)^k(x_1^{p-1}+x_2^{p-1} ) & \ldots &   x_1^{p-1-k}x_2^k \\
\vdots & \vdots & & & & \vdots \\
x_2^{p-1} & x_1x_2^{p-2} & \ldots & x_1^{k}x_2^{p-1-k} & \ldots & x_1^{p-1}+x_2^{p-1}
 \end{array}\right]\left[ \begin{array}{c}h_0 \\ h_1 \\  h_2 \\ \vdots \\ h_k \\ \vdots \\ h_{p-1} \end{array}\right].$$ The $i$-th row represents the coefficient of $y_1^{p-1-i}y_2^i$, and the $k$-th column corresponds to $v_k$. Call this matrix $\mathbf{A}$, denote the vector with entries $h_i$ by $\overrightarrow{h}$ and the vector with entries $f_i$ by $\overrightarrow{f}$. The system of equations in matrix form can then be written as $$\mathbf{A}\overrightarrow{h}=Q_l\overrightarrow{f} .$$

Next, we need a lemma. 
\begin{lemma}
$\det \mathbf{A}=(-1)^{(p-1)/2}Q_1$.
\end{lemma}
\begin{proof}
Factoring out the coefficient $-1$ from all even rows accounts for the sign $(-1)^{(p-1)/2}$. 
Direct computation shows that $\det \mathbf{A}'$ is invariant under $GL_2(\F_p)$ action, and its degree is $p(p-1)$. Hence, it is a multiple of $Q_1$. The coefficient of $x_1^{p(p-1)}$ in both the determinant and $Q_1$ is equal to $1$, which completes the proof. 
\end{proof}

We return to the proof of the proposition and to the equation
$\mathbf{A}\overrightarrow{h}=Q_{l}\overrightarrow{f}.$

For $l=1$, writing $\tilde{\mathbf{A}}$ the adjugate matrix to $\mathbf{A}$ and using $\mathbf{A}^{-1}=\frac{1}{\det \mathbf{A}}\tilde{\mathbf{A}}$, we get $$\overrightarrow{h}=\frac{(-1)^{(p-1)/2}}{Q_1}\tilde{\mathbf{A}}Q_{1}\overrightarrow{f}= (-1)^{(p-1)/2}\tilde{\mathbf{A}}\overrightarrow{f}.$$ So, for every polynomial $f$ there exist polynomials $h_0,\ldots h_{p-1}$ satisfying  (\ref{star}). Picking $f\in S^{0}\h^*\otimes S^{p-1}\h$, it follows that $Q_{1}\otimes S^{p-1}\h$ is contained in $V$.



If $l=0$, then using that $Q_0$ and $Q_1$ are algebraically independent and that $\det \mathbf{A}=(-1)^{(p-1)/2}Q_1$, it follows that if  $\overrightarrow{h}$ and $\overrightarrow{f}$ have polynomial entries and satisfy $\mathbf{A}\overrightarrow{h}=Q_{0}\overrightarrow{f}$, then every entry of $\overrightarrow{h}$ is divisible by $Q_0$. From this it follows that the intersection of $V$ and $\Bbbk Q_0\otimes S^{p-1}$, consisting of vectors of the form (\ref{star}), is the submodule generated by $Q_{0}v_0,\ldots Q_0v_{p-1}$ in degree $(p^2-1)(p-1)$.

\end{proof}

As explained above, the purpose of proving the previous proposition was to conclude:

\begin{cor}\label{charm}Let $M$ be the quotient of $M_{0,c}(S^i\h)$ by the $H_{0,c}(GL_2(\F_p),\h)$-submodule generated by singular vectors $v_0,\ldots v_{p-1}$ in degree $p-1$, $Q_1\otimes S^i\h$ in degree $p^2-p$ and $Q_0\otimes S^i\h$ in degree $p^2-1$. Then its character is
$$\chi_{M}(z)=\chi_{M_{0,c}(S^{p-1}\h)}(z)(1-z^{p-1})(1-z^{p^2-1})$$
and its Hilbert series is a polynomial
$$\mathrm{Hilb}_{M}(z)=p\frac{(1-z^{p-1})(1-z^{p^2-1})}{(1-z)^2}.$$
\end{cor}

\begin{prop}\label{0,p-1,final}
$L_{0,c}(S^{p-1}\h)=M.$  
\end{prop}
\begin{proof}
By Lemma \ref{blocks}, the irreducible representation $L_{0,c}(S^{p-1}\h)$ forms a block of size one. That means that all the irreducible composition factors that appear in the decomposition of $M_{0,c}(S^{p-1}\h)$ and of $M$, are isomorphic to $L_{0,c}(S^{p-1}\h)[m]$. 

As a consequence, the character of $L_{0,c}(S^{p-1}\h)$ divides the character of $M$; 
$$\chi_{L_{0,c}(S^{p-1}\h)}(z)F(z)=\chi_M(z),$$ for some polynomial $F(z)$ with positive integer coefficients. The character of $L_{0,c}(S^{p-1}\h)$ is of the form 
$$\chi_{L_{0,c}(S^{p-1}\h)}(z)=\chi_{M_{0,c}(S^{p-1}\h)}(z)\overline{h}(z)$$ 
 for some polynomial $\overline{h}$ with integer coefficients ($\overline{h}$ is divisible by $(1-z)^2$, as $L_{0,c}(S^{p-1})$ is finite dimensional and $M_{0,c}(S^{p-1})$ has quadratic growth). Substituting this and the character formula for $M$ in the above equation, we get that 
$$\overline{h}(z)F(z)=(1-z^{p-1})(1-z^{p^2-1}).$$

Let us define another version of the character which will enable us to compute $\overline{h}$. For $V=\sum_{k}V_{k}$ a graded Cherednik algebra module, we define $\widetilde{ch}_{V}$ to be a function of a formal variable $z$ and of a group element $g$, defined as
$$\widetilde{ch}_{V}(z,g)=\sum_{k} z^k \mathrm{tr}|_{V_k}(g).$$ It is then easy to see that
$$\widetilde{ch}_{M_{0,c}(S^{p-1}\h)}(z,g)=\mathrm{tr}|_{S^{p-1}\h}(g)\cdot \frac{1}{\det_{\h^*} (1-zg)},$$ 
so 
$$\widetilde{ch}_{M}(z,g)=\mathrm{tr}|_{S^{p-1}\h}(g)\cdot \frac{(1-z^{p-1})(1-z^{p^2-1})}{\det_{\h^*} (1-zg)}$$ $$\widetilde{ch}_{L_{0,c}(S^{p-1}\h)}(z,g)=\mathrm{tr}|_{S^{p-1}\h}(g)\cdot \frac{\overline{h}(z)}{\det_{\h^*} (1-zg)}.$$ 

Let $g\in GL_{2}(\F_p)$. It can be put to Jordan form over a quadratic extension $\F_q$ of $\F_p$, and assume it is diagonalizable with different eigenvalues, of the form $\mathrm{diag}(\lambda, \mu)$ 
with $\lambda\ne \mu \in \F_q$. Then $$\mathrm{tr}|_{S^{p-1}\h}(g)=\lambda^{p-1}+\lambda^{p-2}\mu+\ldots \mu^{p-1}=\frac{\lambda^p-\mu^p}{\lambda-\mu}\ne 0$$ and $\widetilde{ch}_{L_{0,c}(S^{p-1}\h)}(z,g)$ is a polynomial in $z$, so $$\frac{\overline{h}(z)}{\det_{\h^*} (1-zg)}=\frac{\overline{h}(z)}{(1-z\lambda^{-1})(1-z\mu^{-1})}$$ must be a polynomial in $z$ as well. By choosing all possible $\lambda$ and $\mu$ in $\F_p\subseteq \F_q$, this implies that $\overline{h}(z)$ is divisible by all linear polynomials of the form $1-z\lambda^{-1}$, and hence by their product $1-z^{p-1}$. If $\lambda$ and $\mu$ are in the extension $\F_q$ and not in $\F_p$, then the product $(1-z\lambda^{-1})(1-z\mu^{-1})$ is an irreducible quadratic polynomial with coefficients in $\F_p$ with a constant term $1$. All such polynomials can be obtained in this way, and $\overline{h}(z)$ is divisible by their product $(1-z^{p^2-1})/(1-z^{p-1}).$ From this we conclude that $\overline{h}(z)$ is divisible by $1-z^{p^2-1}.$ 

Let us write $$\overline{h}(z)=(1-z^{p^2-1})\phi(z)$$ for some polynomial $\phi$. Then  $$\phi(z)F(z)=1-z^{p-1}.$$ However, it follows from Lemma \ref{hbegins} that $\overline{h}$ is of the form $1-z^{p-1}+\ldots$, so $\phi(z)$ is of that form as well, and it follows that $\phi(z)=1-z^{p-1}$, $F(z)=1$ and $L_{0,c}(S^{p-1}\h)=M$.

\end{proof}

\subsection{Characters of $L_{t,c}(S^{i}\h)$ for $i=p-1$ and $t=1$}

Computing the character of $L_{1,c}(S^{p-1}\h)$ is very similar to computing the character of $L_{0,c}(S^{p-1}\h)$ in the previous section. We define a set of vectors analogous to $v_i$:
$$v'_{k}=\sum_{i=0}^{k}(-1)^i x_1^{p(k-i)}x_2^{p(p-1-k+i)}\otimes y_1^{p-1-i}y_2^i+\sum_{i=k}^{p-1}(-1)^i x_1^{p(p-1+k-i)}x_2^{p(i-k)}\otimes y_1^{p-1-i}y_2^i.$$

\begin{lemma}
The space $span_{\Bbbk}\{v'_0,\ldots, v'_{p-1}\} \subseteq S^{p(p-1)}\h^*\otimes S^{p-1}\h\cong M_{1,c}(S^{p-1}\h)_{p(p-1)}$ consists of singular vectors, and isomorphic to $S^{p-1}\h$ as a $GL_2(\F_p)$ representation. This is the only space of singular vectors in $M_{1,c}(S^{p-1}\h)_{p\cdot i}$ for $i=1,\ldots p-1$.
\end{lemma}
\begin{proof}
The proof that they are singular is an explicit computation analogous to the one in the proof of Lemma  \ref{oneissing}, showing that one vector from this irreducible representation of $GL_{2}(\F_p)$ is singular. The space spanned by them is only space of $p$-th powers in degrees $p,2p,\ldots , (p-1)p$ which is isomorphic to $S^{p-1}\h$ as a $GL_2(\F_p)$ representation; this follows directly from Lemma \ref{hbegins} and implies that this is the only space of singular vectors for generic $c$ in degrees up to $p(p-1)$.


\end{proof}

\begin{prop}\label{1charm}
Let $M'$ be the quotient of $M_{1,c}(S^i\h)$ by the $H_{1,c}(GL_2(\F_p),\h)$-submodule generated by singular vectors $v'_0,\ldots v'_{p-1}$ in degree $p(p-1)$, $Q_1^p \otimes S^i\h$ in degree $p(p^2-p)$ and $Q_0^p \otimes S^i\h$ in degree $p(p^2-1)$. Its character and  Hilbert polynomial are: 
$$\chi_{M'}(z)=\chi_{M_{1,c}(S^{p-1}\h)}(z^p)(1-z^{p(p-1)})(1-z^{p(p^2-1)})\left(\frac{1-z^p}{1-z}\right)^2$$
$$\mathrm{Hilb}_{M'}(z)=p\frac{(1-z^{p(p-1)})(1-z^{p(p^2-1)})}{(1-z)^2}.$$
 \end{prop}
 \begin{proof}
 The claim is equivalent to the reduced character being equal to 
 $$\chi_{M_{1,c}(S^{p-1}\h)}(z)(1-z^{p-1})(1-z^{p^2-1}).$$ By definition of $M'$ and the reduced character, it is equal to the character of the $S\h^*$-module defined as the quotient of $S\h^* \otimes S^{p-1}\h$ by $v_0,\ldots v_{p-1}$ from the previous section, $Q_0 \otimes S^i\h$ and $Q_1 \otimes S^i\h$. Corollary \ref{charm} in the previous section shows that the character of this module is as claimed in the proposition. 
 
 \end{proof}
  
 Finally, we have
\begin{prop}\label{1,p-1,final}
For generic $c$, $L_{1,c}(S^{p-1}\h)=M'$.
\end{prop}
\begin{proof}
The character of $L_{1,c}(S^{p-1}\h)$ for generic $c$ is of the form 
$$\chi_{L_{1,c}(S^{p-1}\h)}(z)=\chi_{M_{1,c}(S^{p-1}\h)}(z^p)\left(\frac{1-z^p}{1-z}\right)^2 \overline{h'}(z^p)$$ for some polynomial $\overline{h'}$. It divides the character of $M'$, so $\overline{h'}(z)$ divides $(1-z^{p-1})(1-z^{p^2-1})$. Using the same version of the character as in the proof of Proposition \ref{0,p-1,final}, we see that 
 $$\widetilde{ch}_{L_{1,c}(S^{p-1}\h)}(z,g)=\mathrm{tr}|_{S^{p-1}\h}(g)\cdot \frac{\overline{h}'(z^p)}{\det_{\h^*} (1-zg)},$$ and we see that $h(z^p)$ is divisible by $(1-z^{p^2-1})$. From this it follows that $h(z^p)$ is divisible by $(1-z^{p(p^2-1)})$. Finally, it follows from the previous proposition that $\overline{h'}(z)$ is of the form $1-z^{p-1}+\ldots$, and from this, its divisibility by  $(1-z^{p^2-1})$, and the fact that it divides $(1-z^{p^2-1})(1-z^{p-1})$, it follows that $\overline{h'}(z)=(1-z^{p^2-1})(1-z^{p-1})$ and $L_{1,c}(S^{p-1}\h)=M'$.

\end{proof}

\section*{Acknowledgments}
We are very grateful to Pavel Etingof for suggesting the problem and devoting his time to it through numerous helpful conversations. We thank Stephen Donkin for the conversation about modular representation theory and for pointing the reference \cite{glover1} to us. This project started as a part of MIT's SPUR and UROP programs for undergraduate research, and was partially funded by them. The work of H.C. was partially supported by the Lord Foundation through the UROP grant, and the work of M.B. was partially supported by the NSF grant  DMS-0758262.

\bibliographystyle{plain}
\bibliography{sources}

\end{document}